\documentclass[10pt]{article}
\usepackage{amssymb}
\usepackage{amsmath}
\usepackage{amsthm}
\usepackage{hyperref}
\usepackage{bbm} 
\usepackage{thmtools, thm-restate} 
\usepackage{enumitem}   
\usepackage{fullpage} 

\usepackage[scaled=.98,sups]{XCharter}
\usepackage[charter,scaled=1.07]{newtxmath} 
\usepackage[cal=cm,frak=mma]{mathalfa} 

\usepackage{titlesec}
\titleformat{\section}
  {\normalfont\scshape\filcenter}{\thesection}{1em}{}
\titleformat{\subsection}
  {\normalsize\scshape}{\thesubsection}{1em}{}

\usepackage{xcolor}
\definecolor{mint}{RGB}{216,254,246}
\definecolor{mintborder}{RGB}{140,220,210}
\definecolor{grayblueborder}{RGB}{100,150,150}
\definecolor{coolgrayborder}{RGB}{120,160,160}
\definecolor{tealborder}{RGB}{0,128,128}
\usepackage[textsize=tiny, color=mint, bordercolor=mintborder,disable]{todonotes} 

\usepackage{kotex}

\renewenvironment{abstract}
  {\begin{center}\begin{minipage}{0.85\linewidth}\vspace{-0.5em}\noindent\textsc{Abstract.} \ignorespaces}
  {\end{minipage}\end{center}}

\numberwithin{equation}{section}

\newtheorem{theorem}{Theorem}

\newtheorem{lemma}{Lemma}
\newtheorem{proposition}{Proposition}

\newtheorem{remark}{Remark}

\newcommand{\nrm}[1]{\Vert#1\Vert}

\newcommand{\tld}[1]{\widetilde{#1}}

\newcommand{\dist}{\mathrm{dist}\,}

\newcommand{\supp}{{\mathrm{supp}}\,}

\newcommand{\diam}{\mathrm{diam}\,}

\newcommand{\aleq}{\lesssim}
\newcommand{\ageq}{\gtrsim}

\newcommand{\ud}{\mathrm{d}}
\newcommand{\rd}{\partial}

\newcommand{\alp}{\alpha}

\newcommand{\gma}{\gamma}

\newcommand{\dlt}{\delta}
\newcommand{\Dlt}{\Delta}

\newcommand{\veps}{\varepsilon}

\newcommand{\tht}{\theta}

\newcommand{\omg}{\omega}
\newcommand{\Omg}{\Omega}
 
\newcommand{\zt}{\zeta}

\newcommand{\vphi}{\varphi}

\newcommand{\bfB}{{\bf B}}

\newcommand{\bfO}{{\bf O}}

\newcommand{\bbN}{\mathbb N}

\newcommand{\bbR}{\mathbb R}

\newcommand{\bbT}{\mathbb T}

\newcommand{\calA}{\mathcal A}

\newcommand{\calD}{\mathcal D}

\newcommand{\calF}{\mathcal F}

\newcommand{\calH}{\mathcal H}

\newcommand{\calK}{\mathcal K}
\newcommand{\calL}{\mathcal L}

\newcommand{\calN}{\mathcal N}

\newcommand{\calQ}{\mathcal Q}

\newcommand{\calT}{\mathcal T}

\newcommand{\calV}{\mathcal V}

\newcommand\blfootnote[1]{%
  \begingroup
  \renewcommand\thefootnote{}\footnote{#1}%
  \addtocounter{footnote}{-1}%
  \endgroup
}

\author{\normalsize {SEUNGJAE LEE}}
\date{}

\begin{document}

    \title{\large\textbf{SHARP LOCAL WELL-POSEDNESS OF $C^1$ VORTEX PATCHES}}
    \date\today

    \maketitle

    \blfootnote{\noindent Seungjae Lee: Department of Mathematical Sciences, Seoul National University, South Korea, email: uhf0125@snu.ac.kr 
     }

    \begin{abstract}
    It is well known that the boundary dynamics of vortex patches is globally well-posed in the H\"older space $C^{1,\alpha}$ for $0<\alpha<1$, whereas the well-posedness in $C^1$ remains an open problem, even locally. In this paper, we establish the local well-posedness for vortex patches in the space $C^{1,\varphi}$ defined via a modulus of continuity $\varphi$ that satisfies certain structural assumptions. Our class includes curves that are strictly rougher than the H\"older-continuous ones, with prototypical examples being $\varphi(r) = (-\log r)^{-s}$ for $s>3$. Motivated by the fact that the velocity operator in the contour dynamics equation is a nonlinear variant of the Hilbert transform, we study the system of equations satisfied by the curve parametrization $\gamma \in C^{1,\varphi}$ and its Hilbert transform. In doing so, we derive several properties of the Hilbert transform and its variants in critical spaces, which are essential for controlling the velocity operator and its Hilbert transform.
    \end{abstract}

\tableofcontents

\newpage
\section{Introduction} \label{Sec:Intro}

\subsection{Vortex Patches}

We consider the two-dimensional Euler equations in vorticity form:
\begin{align} \label{eq:EE} \tag{EE}
    \begin{cases}
        & \rd_t \omg + (u\cdot\nabla) \omg = 0 \\
        & u = -\nabla^\perp (-\Dlt)^{-1} \omg \\
        & \omg_{t=0} = \omg_0.
    \end{cases}
\end{align}
A vortex patch is a weak solution to \eqref{eq:EE} of the form
\begin{align*}
    \omg(t,x) = k\mathbbm{1}_{\Omg(t)},
\end{align*}
where $\mathbbm{1}_{\Omg(t)}$ denotes the characteristic function of a connected bounded domain $\Omg(t) \subset \bbR^2$ and $k\in\bbR$ denotes the strength of vorticity. In this paper, we consider only a single patch in $\bbR^2$ with constant vorticity $k=2\pi$.

Due to the transport structure of the two-dimensional Euler equations, the vortex patch dynamics reduces to the evolution of its boundary. If the patch boundary is at least piecewise $C^1$, one can derive a one-dimensional equation describing the parametrization of the boundary $\gma : \bbT\times [0,T] \to \bbR^2$, known as the contour dynamics equation \cite{ZHR79, Maj86}: 
\begin{align} \label{eq:CDE} \tag{CDE}
    \begin{cases}
        \displaystyle & \rd_t \gma(\xi,t) = \int_\bbT \rd_\xi \gma(\eta,t) \log|\gma(\xi,t)-\gma(\eta,t)| \ud \eta \\
        & \gma(\xi,0) = \gma_0(\xi),
    \end{cases}
\end{align}

On the other hand, given any patch-type initial data $\omg_0 = \mathbbm{1}_{\Omg_0}$, the Yudovich theory \cite{Yud63} guarantees the existence and uniqueness of the patch-type global solution $\omg(t) = \mathbbm{1}_{\Omg(t)}$ for some bounded domain $\Omg(t) \subset \bbR^2$. However, it does not provide any information on the regularity of $\rd\Omg(t)$. The question of the boundary regularity propagation naturally leads to the well-posedness theory of \eqref{eq:CDE}.

Bertozzi \cite{Ber91} proved the local well-posedness of $C^{k,\alp}$ vortex patches by using the \eqref{eq:CDE}. Her proof is based on the potential theory estimates of the velocity operators 
\begin{align*}
    v[\gma](\xi) := \int_\bbT \rd_\xi \gma(\eta) \log|\gma(\xi)-\gma(\eta)| \ud \eta.
\end{align*}
Also, Bertozzi and Constantin \cite{BC93} and Chemin \cite{Che93} achieved the global well-posedness of the smooth vortex patches by using the two-dimensional dynamics of the Euler equations, without using \eqref{eq:CDE}. The question of the ill-posedness of $C^k$ vortex patches for $k\in \bbN$ remains open until Kiselev and Luo \cite{KL23} proved the instantaneous blow-up of $C^2$ vortex patches exploiting the equations for curvature of boundary and using the dispersive nature of velocity operator. However, the question on the instantaneous blow-up of $C^1$ vortex patches is still open.

This work aims to establish the local well-posedness of vortex patches in spaces that lie between $C^1$ and $C^{1,\alp}$. We consider spaces $C^{1,\vphi}$ defined via a modulus of continuity $\vphi$ that satisfies a set of structural assumptions. These assumptions identify a precise boundary for our method, extending the theory to a class of moduli strictly rougher than any Hölder moduli while retaining the minimal regularity required for the key analytic estimates.

\subsection{Main Results} \label{Subsec:Main Results}

The main challenge in extending the local well-posedness theory to $C^{1,\vphi}$ spaces lies in the failure of the classical potential theory estimates, which are central to the $C^{1,\alp}$ case. Indeed, these estimates were the cornerstone of Bertozzi's work in \cite{Ber91}, where they were used to establish the boundedness of the velocity operator. In general $C^{1,\vphi}$ spaces with moduli $\vphi$ strictly rougher than any Hölder moduli, this boundedness is lost, and the operator is no longer continuous.
 
The key to overcoming this difficulty lies in the structural properties of the velocity operator, particularly its underlying structure as a nonlinear variant of the Hilbert transform. Specifically, we establish that 
\begin{align*}
    v[\gma] \in C^{1,\vphi} \text{ provided that } \gma, \calH[\gma] \in C^{1,\vphi},
\end{align*}
where $\calH[\gma]: = (\calH[\gma_1],\calH[\gma_2])^t$. Motivated by this, we study the system of equations satisfied by the parametrization $\gma$ together with its Hilbert transform $\calH[\gma]$. That is, we consider the following system:
\begin{align} \label{eq:HCDE} \tag{HCDE}
    \begin{cases}
        & \rd_t \gma = v[\gma] \\
        & \rd_t \calH(\gma) = \calH[v[\gma]] \\
        & \gma_{t=0} = \gma_0 \\
        & \calH[\gma]_{t=0} = \calH[\gma_0].
    \end{cases}
\end{align}

We now introduce the assumptions on the modulus $\vphi$. We consider moduli $\vphi$ satisfying the following assumptions:
\begin{enumerate}[label=(\kern0.05em A\arabic*)] 
    \item\label{A1} There exists $0 < \tht < \frac{1}{2}$ such that 
    \begin{align*} 
        \int_0^r \frac{\vphi^\tht(r')}{r'} \ud r' \to 0 \quad \text{as } r \to 0.
    \end{align*}
    \item\label{A2} There exists $C > 0$ such that for any $\delta < 1$, 
    \begin{align*} 
        \sup_{r<C\delta} \frac{\vphi(\delta^{-1} r)}{\vphi(r)}\vphi(\delta) < \infty.
    \end{align*}
    \item\label{A3} It holds that 
    \begin{align*} 
        \vphi(r) \left( -\log r \right) \in C^{\tld{\vphi}},
    \end{align*}
    where 
    \begin{align} \label{eq:def_induced modulus}
        \tld{\vphi}(r) : = \int_0^{r} \frac{\vphi(r')}{r'} \ud r'.
    \end{align}
    From now on, we refer to $\tld{\vphi}$ as the induced modulus of $\vphi$.
\end{enumerate}
The class of moduli satisfying the assumptions includes moduli that are strictly rougher than the H\"older modulus. The prototypical examples are $\vphi(r)=(-\log r)^{-s}$, $s > 2$ with the induced modulus $\tld{\vphi}(r) = \frac{1}{s-1}(-\log r)^{s-1}$. 

Our main results establish the local well-posedness result of \eqref{eq:HCDE} in the spaces $C^{1,\vphi}$, where $\vphi$ satisfies the assumptions \ref{A1}--\ref{A3}.
\begin{theorem} \label{Thm1}
    Let $\vphi$ be a modulus of continuity satisfying assumptions \ref{A1}--\ref{A3} and $\gma_0, \calH[\gma_0] \in C^{1,\vphi}(\bbT)$. We further assume that $\gma$ is a non-degenerate parametrization:
    \begin{align*}
        |\gma_0|_\star = \inf_{\xi\neq \eta} \frac{|\gma_0(\xi)-\gma_0(\eta)|}{|\xi-\eta|} > 0.
    \end{align*} 
    Then there exists $T = T(\vphi, \nrm{\gma_0}_{C^{1,\vphi}}, \nrm{\calH[\gma_0]}_{C^{1,\vphi}},|\gma_0|_\star)$ and the unique solution $(\gma(t),\calH[\gma](t)) \in C^0([0,T];C^{1,\vphi}(\bbT))^2$ to \eqref{eq:HCDE}, satisfying $|\gma(t)|_\star>0$ for $t\in[0,T]$.
\end{theorem}

As a simple consequence of theorem \ref{Thm1}, we establish the following local well-posedness of vortex patches in the critical spaces.
\begin{theorem} \label{Thm2} 
    Let $\psi_0$ be a modulus of continuity and let the initial vorticity be $\omg_0 = \mathbbm{1}_{\Omg_0}$ where $\Omg_0 \subset \bbR^2$ is a $C^{1,\psi_0}$ domain. We further assume that induced modulus $\psi:=\tld{\psi_0}$ defined as 
    \begin{align*} 
        \psi(r) = \int_0^r \frac{\psi_0(r')}{r'} \ud r',
    \end{align*}
    satisfies \ref{A1}--\ref{A3}. Then there exists $T = T(\Omg_0)$ and the unique patch solution $\omg(t) = \mathbbm{1}_{\Omg(t)}$ with initial data $\omg_0$ such that $\Omg(t)$ is a $C^{1,\psi}$ domain for $t \in [0,T]$.
\end{theorem}

\begin{remark} [Remarks on the Main Theorem]
    \phantom{=}
    \begin{enumerate}
        \item Note that the parametrization $\gma$ can lose its regularity instantaneously, while it persists within the lower bound of the regularity in short time. On the other hand, the Theorem \ref{Thm2} ensures that there is no instantaneous $C^1$ blow-up for initial $C^{1,\vphi}$ patches. For example, for an initial data $\Omg_0$ which is $C^{1,(-\log \cdot)^s}$ domain where $s>3$, $\Omg(t)$ remains $C^{1,(-\log \cdot)^{s-1}}$ domain for $t \in [0,T]$ for some $T>0$.
        
        \item The key difference between our results and $C^{1,\alpha}$ well-posedness is that the singular integral operators are, in general, not bounded on $C^{1,\vphi}$. Thus the regularity assumption on $\calH[\gma]$ is essential for obtaining the estimates on the velocity operator. This approach additionally requires estimates on the Hilbert transform of the velocity operator, for which obtaining an exact formula is essential. These issues are discussed in detail in Section \ref{Subsec:Ideas} and in the proof of Proposition \ref{Prop:est of Hv}.
        
        \item The problem of global well-posedness faces several obstacles, which we discuss in Section \ref{Sec:Discussions}.
    \end{enumerate} 
\end{remark}

\begin{remark} [Remarks on the assumptions on moduli]
    \phantom{=}
    \begin{enumerate}
        \item We refer to $\tld{\vphi}$ as the induced modulus of $\vphi$. Generally, the singular integral operators and also the velocity operator $v[\gma]$ are unbounded in spaces $C^{1,\vphi}$ with general modulus $\vphi$. Nevertheless, if $\vphi$ is Dini-continuous, we can still ensure the weaker boundedness of Hilbert transform:
        \begin{align*}
            \calH[C_c^\vphi] \subset C^{\tld{\vphi}}.
        \end{align*}
        Note that $\tld{\vphi}$ is in general strictly rougher than $\vphi$. Thereby motivating the introduction of the induced modulus $\tld{\vphi}$ - see Lemma \ref{Lem:img of H}.

        \item The condition \ref{A1} ensures that functions in $C^\vphi$ are Dini-continuous. It also allows us to use an interpolation inequality. For a function $f \in C^\vphi$, we have 
        \begin{align} \label{eq:interpolation}
            \nrm{f}_{C^{\vphi^\tht}} \leq \nrm{f}_{C^0}^{1-\tht} \nrm{f}_{C^\vphi}^\tht.
        \end{align}

        \item The condition \ref{A2} is required to ensure an appropriate scaling property of the space $C^\vphi$.
        
        \item The condition \ref{A3} is used to control the boundary terms that arise when performing integration by parts in the Hilbert transform. 
        
    \end{enumerate}
\end{remark}

\subsection{Historical Backgrounds} \label{Subsec:Backgrounds}
\noindent\textbf{$\bullet$ Regularity of Vortex Patches}

The vortex patch problem for the 2D Euler equations is a fundamental and extensively studied problem with a rich literature. As we mentioned before, Bertozzi proved the local well-posedness of \eqref{eq:CDE} in the space $C^{k,\alpha}$. In \cite{Che93} and \cite{BC93}, Chemin, Bertozzi and Constantin proved that an initially smooth domain remains smooth globally. The case of $C^k$ remained open until Kiselev and Luo proved the ill-posedness of vortex patches in $C^k$ for $k \geq 2$. However, it is still unknown whether the vortex patch problem is well-posed in $C^1$. 

\noindent\textbf{$\bullet$ Dynamics of Vortex Patch Evolution}

In \cite{BH10}, Biello and Hunter exhibited that the evolution of the interfaces between two half-spaces with constant vorticities in $\bbR^2$ coincides with the asymptotic equation of Burgers-Hilbert equation. Also, in \cite{KL23} Kiselev and Luo showed that the evolution of curvature is governed by the Hilbert transform with smooth perturbation terms.

In recent years, there has been substantial progress on the existence and bifurcation theory of V-state - uniformly rotating vortex patches - see \cite{DZ78} \cite{Bur82} \cite{HMV13}. These results suggest that the evolution of the patch boundary exhibits a dispersive nature. Altogether, these findings support that the velocity operator is closely related to the Hilbert transform.

\noindent\textbf{$\bullet$ Finite time Singularity Formation}

One may ask the finite time singularity formation for $C^{1,\vphi}$ boundary. One can construct an initial parametrization $\gma_0$ such that its solution to \eqref{eq:CDE} instantaneously leaves $C^{1,\vphi}$. However, the key difficulty in the ill-posedness problem of vortex patches lies in the fact that it is possible for the curve to be $C^{1,\vphi}$ even if its parametrization is not. 

It is still unknown whether the vortex patch problem is well-posed in $C^1$. Since there is a parametrization issue in the $C^1$ class, one may consider the 2D dynamics of the Euler equations. For an initially $C^1$ curve, the cusp formation is considered as a strong candidate for blow-up.

Let us examine the cusp formation scenario. When the initial curve has a corner, its evolution has been numerically studied in \cite{CD00}, \cite{CS99}, and \cite{CS00}. All three works suggest that an acute corner tends to form a cusp, whereas an obtuse corner tends to flatten out. Moreover, Elgindi and Jeong \cite{EJ23} rigorously proved these predictions rigorously by exploiting n-fold symmetry to fix the moving axis, and more recently Elgindi and Jo \cite{EJ25} established the formation of cusps from acute corners even in a moving frame. Given these results, it appears heuristically difficult for an initially smooth curve to develop a cusp in finite time. Intuitively, during the process of cusp formation, the curve must transiently resemble a corner. Assuming the evolution is continuous, a smooth curve would first form an obtuse corner, which would then flatten rather than progress towards a cusp.

\noindent\textbf{$\bullet$ The 2D Euler equations in critical spaces}

There is a substantial body of literature concerning the propagation of regularity for the 2D Euler equations in critical spaces. In \cite{Koc02}, Koch proved that if the initial vorticity is Dini continuous, then its regularity propagates. Note that the function spaces $C^{\vphi}$ consist of Dini-continuous functions. 

In \cite{CJ23}, Chae and Jeong proved that the propagation of regularity of vorticity in $(-\log r)^{-s}$ with $0<s<1$ under the assumption that the velocity field is Lipschitz. Recently in \cite{Kha24} Kharim proved the same regularity propagation results without the assumption on velocity field. In addition, he constructed the family of moduli, which alternates between the two moduli $\log(-\log r)(-\log r)^{-s}$ and $(-\log r)^{-s}$ for $0<s<1$, such that loses its regularity in finite time. As in Theorem \ref{Thm1}, our well-posedness class of \eqref{eq:CDE} include the space $C^{1,\vphi}$ with $\vphi(r) = (-\log r)^s$ for $s < -2$. Our result is similar to the results of Koch and Kharim.

However, there are a few different features between the 2D Euler equations and vortex patch evolution. The Euler equations have the transport structure. Hence the regularity propagation is closely related to the regularity of velocity field - see \cite{Yud63}. Consequently, the local well-posedness issue reduces to the analysis on the Biot-Savart kernel. In the ill-posed problem, the hyperbolic scenario can be useful. On the other hand, the contour dynamics equation itself has a Hilbert-transform-type operator, which can exhibit the dispersive properties. It is interesting that the Euler equations - origin of contour dynamics equation - have transport structure, whereas the contour dynamics equation does not.

\noindent\textbf{$\bullet$ Sadovskii Vortex Patch}

The Sadovskii vortex patch represents a patch-type traveling wave solution to the two-dimensional Euler equations, characterized by odd symmetry and touching the symmetry axis. Its existence was first observed by Sadovskii in \cite{SAD71} through numerical simulations.

In 2024, Huang and Tong \cite{HT24}, and Choi, Jeong and Sim \cite{CJS24}, independently proved the existence of Sadovskii patch using the different methods. It is still unknown whether the patches in \cite{HT24} and \cite{SAD71} are the same, and the uniqueness of the Sadovskii patch also remains open. Interestingly, the Sadovskii patch constructed by Huang and Tong contains a right-angle corner with the modulus $(-\log r)^{-1}$, which is critical in terms of the boundedness of singular integral operators.

\noindent\textbf{$\bullet$ Free boundary problem}

Recently, there has been substantial research on the free boundary problem in critical regularity spaces. In \cite{GP21}, Gancedo and Patel proved the local well-posedness of $\alpha$-SQG patches in the space $H^2(\bbT)$ for $0<\alpha<1/3$. They also presented the finite-time singularity formation in the space $H^2$ and blow-up criterion regarding the arc-chord for the same $\alpha$. 

The Muskat problem \cite{Mus34} describes the dynamics of interface between two fluids in porous media. In \cite{CL18}, Cordoba and Lazar proved the global well-posedness for the 2D Muskat problem in $H^{\frac{3}{2}}$. In \cite{GL20}, Gancedo and Lazar proved the global well-posedness for the 3D Muskat problem in $\dot{H}^2 \cap \dot{W}^{1,\infty}$. 

\subsection{Ideas of Proof} \label{Subsec:Ideas}

Our proof is based on sharp estimates of the velocity operator and its Hilbert transform in critical spaces. The velocity operator is closely related to the Hilbert transform, which is one of the simplest singular integral operator - see \cite{BH10}, \cite{KL23}. We begin by considering the properties of Hilbert transform, which serves as a toy model for the velocity operator.

\noindent\textbf{$\bullet$ Hilbert transform on the critical spaces}

It is well-known that the Hilbert transform is bounded on $\dot{C}^{k,\alp}$ for $k \in \bbN, 0<\alp<1$ and unbounded on $\dot{C}^k$. In spaces $C^{k,\vphi}$ with moduli $\vphi$ strictly rougher than any H\"older moduli $r^\alp$, the Hilbert transform is generally unbounded and it is bounded if and only if $\vphi \in C^\alp$ for some $0<\alp<1$. More generally, in Lemma \ref{Lem:img of H}, we establish that for a Dini-continuous modulus $\vphi$, the Hilbert transform of a $C^{k,\vphi}$ function belongs at least to $C^{k,\tld{\vphi}}$ where $\tld{\vphi}$ is the induced modulus of $\vphi$ defined in \eqref{eq:def_induced modulus}: 
\begin{align*}
    \tld{\vphi}(r) = \int_0^r \frac{\vphi(x)}{x} \ud x.
\end{align*}
Note that $\tld{\vphi}$ is strictly rougher than $\vphi$ if $\vphi \notin C^\alpha$. Therefore, evolution equations involving the Hilbert transform are generally ill-posed in spaces where the Hilbert transform is unbounded. For example, in \cite{EM20}, Elgindi and Masmoudi proved the ill-posedness of such equations in $L^\infty$ based spaces. However, we introduce a method to overcome this loss of regularity.

\noindent\textbf{$\bullet$ Evolution Equations involving the Hilbert transform}

We now focus on the local well-posedness problem of the following equations in $C^\vphi$:
\begin{align*}
    \rd_t f = \calF(f,\calH[f]),
\end{align*}
where $\calF$ is a smooth function. In the H\"older spaces $C^\alp$, thanks to the continuity of the Hilbert transform, the forcing term on the right-hand side belongs to $C^\alp$ as well. Hence standard techniques yield local well-posedness in these spaces $C^\alp$. 

However, the situation changes drastically in $C^\vphi$, where the Hilbert transform is no longer continuous. For instance, one may attempt to prove local well-posedness in $C^\vphi$ via an iterative scheme. Given initial data $f_0 \in C^\vphi$, one constructs a sequence $f_n$ via the standard iteration scheme:
\begin{align*}
    \rd_t f_n = \calF(f_n,\calH[f_{n-1}]).
\end{align*}
However, since $\calH$ is not bounded on $C^\vphi$, we cannot prevent a loss of regularity. That is, we cannot guarantee that $f_n \in C^\vphi$ for any $n > 1$. 

More precisely, what we can ensure is only that $f_n \in \calH^n[C^\vphi]$, a progressively weaker regularity space. One may then attempt to work in the union of spaces $\cup_{n=1}^\infty \calH^n[C^\vphi]$. However, this space is generally not well-defined. For instance, if the initial modulus is $\vphi(r) = (-\log r)^{-s}$ with $s>1$, the $n$-th induced modulus is given by $(-\log r)^{s-n}$ for $s-n>1$. Moreover, for sufficiently large $n$, the image $\calH^n[C^\vphi]$ is not contained in $L^\infty$, which implies that $f_n$ no longer has the meaning of a strong solution. 

To overcome this loss of regularity, one must exploit specific structural properties of the function $\calF$. As an example, we present a method to prove the local well-posedness of the transport-Hilbert equation: 
\begin{align} \tag{TH} \label{eq:TH}
    u_t + cu_x = \calH[u].
\end{align}
This equation can be viewed as the linearization of the Burgers-Hilbert equation with constant speed $c$, which is known to be an approximate model for vortex patch dynamics - \cite{MW83}, \cite{BH10}, \cite{HM22}. Note that the forcing term is linear in $\calH[u]$, and this linear structure allows us to exploit the identity $\calH^2 = -id$. Applying the Hilbert transform to \eqref{eq:TH}, we obtain the following equivalent linear system for $(u,\calH[u])$: 
\begin{align} \tag{HTH} \label{eq:HTH}
    \begin{cases}
        & u_t + c u_x = \calH[u] \\
        & \calH[u]_t + c\calH[u]_x = -u.
    \end{cases}
\end{align}
Using this reformulation, one can establish the local well-posedness of the system \eqref{eq:HTH} in the space $(u,\calH[u])\in (C^{1,\vphi})^2$. 


However, this method cannot be directly applied to our case \eqref{eq:CDE}, whose forcing term is a highly nonlinear variant of the Hilbert transform, as we will demonstrate in the following section. 

\noindent\textbf{$\bullet$ Decomposition on velocity operator}

We now prove Theorem \ref{Thm1}. The proof relies on a fixed-point argument, which follows from key $C^{1,\vphi}$ estimates for the velocity operator $v$ and its Hilbert transform $\calH[v]$. These estimates are obtained by exploiting the structural properties $v$ and $\calH[v]$, which in turn overcome the preceding nonlinear difficulty. Our analysis begins with the velocity derivative
\begin{align*}
    \rd_\xi v[\gma](\xi) = P.V.\int_\bbT \frac{(\gma(\xi)-\gma(\eta))\cdot\gma(\xi)}{|\gma(\xi)-\gma(\eta)|^2} \rd_\xi \gma(\eta) \ud \eta.
\end{align*} 
The primary difficulty lies in its kernel, $\frac{(\gma(\xi)-\gma(\eta))\cdot\gma(\xi)}{|\gma(\xi)-\gma(\eta)|^2}$. In this form, the nature of its singularity is obscure and not immediately comparable to the classic Hilbert kernel $\frac{1}{2}\cot(\frac{\xi-\eta}{2})$.

The main goal of this part is to precisely decompose this kernel to expose its fundamental structure. To reveal the Hilbert kernel $1/(\xi-\eta)$, we can formally factor it out:
\begin{align*}
    \frac{(\gma(\xi)-\gma(\eta))\cdot\gma(\xi)}{|\gma(\xi)-\gma(\eta)|^2} = \frac{1}{\xi-\eta} \frac{\calQ[\gma](\xi,\eta)\cdot \gma(\xi)}{|\calQ[\gma](\xi,\eta)|^2},
\end{align*}
where $\calQ[\gma](\xi,\eta) := \frac{\gma(\xi)-\gma(\eta)}{\xi-\eta}$ denotes the difference quotient of $\gma$. The challenge is now the term $|\calQ[\gma]|^{-2}$. A direct analysis is difficult because the Hilbert transform of reciprocal of a function is generally hard to handle. 

Our strategy is to localize $\gma$ into a linear part $(U \xi)$, for $U = \rd_\xi \gma(\xi)$, and a small nonlinear remainder $\zt$. That is, we consider the vector-valued function $\zt = (\gma-U\xi)\chi$, where $\chi$ is a compactly supported bump function supported on $[-\dlt,\dlt]$. This procedure is designed to make the denominator amenable to a power series expansion. We implement this by assuming $\dlt$ small enough. Then this leads to the following decomposition:
\begin{align*}
    & |\calQ[\gma](\xi,\eta)|^{-2} = |U+\calQ[\zt](\xi,\eta)|^{-2} = |U+\calQ[\zt](\xi,\eta)|^{-2} = |U|^{-2} \sum_{n=0}^\infty (-1)^n\calD^n(\xi,\eta) \\
    & \calD(\xi,\eta) := \frac{2}{|U|}(U\cdot\calQ[\zt](\xi,\eta)) + \frac{1}{|U|^2} |\calQ[\zt](\xi,\eta)|^2,
\end{align*}
where $\calD$ is a quadratic polynomial in $\calQ[\zt]$. Substituting this back, we finally reveal the structure of the kernel: it is an infinite series of terms, each containing the Hilbert kernel $1/(\xi-\eta)$ multiplied by high powers of the difference quotient $\calQ[\zt](\xi,\eta)$. 

Combining the above arguments and after carefully localizing, the estimates of $\rd_\xi v[\gma]$ reduces to dealing with the following operators
\begin{align*}
    S[\zt](\xi) := P.V. \int \sum_{n=0}^\infty \frac{\rd_\xi \zt(\eta)}{\xi-\eta}\calD^n(\xi,\eta) \ud\eta.
\end{align*}

\noindent\textbf{$\bullet$ $C^{1,\vphi}$ Estimates on $v[\gma]$}

We now proceed to obtain a $C^{1,\vphi}$ estimate on $v[\gma]$ under the assumption that both $\gma$ and $\calH[\gma]$ belong to the space $C^{1,\vphi}$. The analysis of the $\rd_\xi v[\gma]$ is central to this proof. The full expression for $\rd_\xi v[\gma]$ consists of a series of complex operators. For the sake of brevity and clarity, we will present the detailed estimate only for the key operator $S[\zt]$, defined as follows. The other terms can be handled by similar arguments. \footnote{The treatment of the difference quotient on the torus is somewhat delicate due to the singular factor $1/(x-y)$. To avoid this difficulty, we extend the functions $\zt$ and $\chi$ from $\bbT$ to the whole space $\bbR$. This procedure is needed for notational and calculational convenience. One may prove the whole procedure regarding $\zt$ as a function on $\bbT$.}
\begin{align*}
    S[\zt](\xi) := P.V.\int_\bbR \sum_{n=0}^\infty \frac{\rd_\xi \zt(\eta)}{\xi-\eta} \calD^n(\xi,\eta) \ud \eta,
\end{align*}
We establish the estimates on $S[\zt]$ by establishing the bounds of the $C^\vphi$ norm of its summand
\begin{align*}
    S_n[\zt](\xi) := P.V.\int_\bbR \frac{\rd_\xi \zt(\eta)}{\xi-\eta} \calD^n(\xi,\eta) \ud \eta,
\end{align*}
and its summability. A simple calculation yields the following decomposition of $S_n$:
\begin{align*}
    S_n[\zt] (\xi)
    & = \rd_\xi \zt(\xi) \calH_2[\calD^n](\xi,\xi) + \calH[\rd_\xi \zt](\xi) \calD^n(\xi,\xi) \\
    & \quad + \int_\bbR \frac{(\rd_\xi \zt(\xi) - \rd_\xi \zt(\eta))(\calD^n(\xi,\xi)-\calD^n(\xi,\eta))}{\xi-\eta} \ud\eta,
\end{align*}
where $\calH_2$ denote the Hilbert transform acting on the second argument. From now on, we estimate each term.

We first examine the norm of $\zt$. Since $\vphi$ is a general modulus, the scaling property of $C^{1,\vphi}$ is not clear, which is required to control the norms of $\zt$ with respect to localizing spatial scale $\dlt$. Thus we introduce the assumptions \ref{A2}, which ensures the following scaling property (see \eqref{eq:scaling of C^1,vphi}):
\begin{align*}
    \nrm{f(\cdot/\dlt)}_{\dot{C}^{1,\vphi}} \aleq \frac{1}{\dlt\vphi(\dlt)}\nrm{f}_{\dot{C}^{1,\vphi}}.
\end{align*} 
Consequently, we obtain the following bounds for $\zt$:
\begin{align*}
    \nrm{\zt}_{C^1} \aleq \nrm{\gma}_{C^1} \vphi(\dlt) ,\quad \nrm{\zt}_{C^{1,\vphi}} \aleq \nrm{\gma}_{C^{1,\vphi}}.
\end{align*}

We now consider the norms of $\calH[\zt] = \calH[\gma \cdot \chi]$. Since we deal with $\calH[\gma \cdot \chi]$ instead of $\calH[\gma]\cdot\chi$, we cannot use the above scaling property directly, and thus we need to exploit the multiplication property of the Hilbert transform on the space $C^{1,\vphi}$. Then one might try to use the well-known Cotlar identity:
\begin{align} \label{eq:Cotlar}
    \calH[fg] = \calH[f]g + f\calH[g] + \calH[\calH[f]\calH[g]].
\end{align}
In the space where the Hilbert transform is continuous, this identity is highly useful: the regularity of $\calH[fg]$ can be deduced from the regularity of $f$ and $g$. However, in the space where $\calH$ is not continuous, we cannot directly apply this argument. Even though we assume the additional condition on $\calH[f]$ and $\calH[g]$, we cannot prevent the loss of regularity due to the last term $\calH[\calH[f]\calH[g]]$.

Nevertheless, we refine the Cotlar identity to the following identity:
\begin{align} \label{eq:Hfg formula}
    \calH[fg](x) = \frac{1}{\pi}P.V.\,\int_\bbR \frac{(\calH[f](x)-\calH[f](y))(\calH[g](x)-\calH[g](y))}{x-y} \ud y.
\end{align} 
This identity allows one to extract a decay of order $\vphi^2$ from the differences of $f$ and $g$, therefore preventing the regularity loss under the Hilbert transform. A detailed proof is given in Lemma \ref{Lem:Hfg}. Also, a potential theory estimate of the above formula together with the interpolation inequality \eqref{eq:interpolation} yields \footnote{In this estimate, the assumptions \ref{A1} and the interpolation inequality \eqref{eq:interpolation} are crucial. Without it, we can only obtain the bound:
\begin{align*}
    \nrm{\calH[fg]}_{C^\vphi} \aleq \nrm{\calH[f]}_{C^\vphi} \nrm{\calH[g]}_{C^\vphi}.
\end{align*}
This weaker inequality is insufficient to establish other key estimates, such as the decay estimate for $\nrm{\calH[\calD^n]}_{C^\vphi}$.
} 
\begin{align} \label{eq:Hfg nrm}
    \nrm{\calH[fg]}_{C^\vphi} \aleq \nrm{\calH[f]}_{C^\vphi} ( \nrm{\calH[g]}_{C^\vphi}^\tht \nrm{\calH[g]}_{C^0}^{1-\tht} + \nrm{\calH[g]}_{C^0} ) +  \nrm{\calH[g]}_{C^\vphi} (\nrm{f}_{C^0} + \nrm{\calH[f]}_{C^0}).
\end{align}
By using this norm estimates, we obtain the following bounds on $\calH[\zt]$:\footnote{We remark that the apparent divergence of the prefactor $\vphi(\dlt)^{\tht-1}$ of $\nrm{\calH[\zt]}_{C^{1,\vphi}}$ as $\dlt \to 0$ is not an issue in our argument. Since the proof only requires satisfying a finite set of smallness conditions on $\dlt$, we can fix a single positive $\dlt=\dlt(M)$ throughout the analysis. Consequently, this prefactor becomes a constant.}
\begin{align*}
    \nrm{\calH[\zt]}_{C^1} \aleq \vphi(\dlt)^{1-\tht} \nrm{\calH[\gma]}_{C^{1,\vphi}},\quad \nrm{\calH[\zt]}_{C^{1,\vphi}} \aleq \vphi(\dlt)^{-\tht} \nrm{\calH[\gma]}_{C^{1,\vphi}}. 
\end{align*}

We now examine the $C^{\vphi}$ bounds for $\calD^n$ and $\calH[\calD^n]$. To ensure summability, we must extract a sufficient rate of decay from these bounds. The summability of $\nrm{\calD^n}_{C^\vphi}$ follows from the algebraic properties of $C^{\vphi}$, provided that $\dlt$ is chosen small enough to satisfy $\nrm{\calD}_{C^0} < 1$. Furthermore, an inductive application of \eqref{eq:Hfg nrm} to $\calH[\calD^{n-1} \calD]$ yields
\begin{align*}
    \nrm{\calH[\calD^n]}_{C^\vphi} \aleq \vphi(\dlt)^{\tht-1} (M\alp_1)^n,
\end{align*}
where $M$ is a constant depends on the norms of $\zt$ and $\calH[\zt]$ and $\alp_1 = C \vphi(\dlt)^{1-2\tht}$. Note that to ensure the summability of $\nrm{\calH[\calD^n]}_{C^\vphi}$, we need the restriction on the range of $\tht$: $0<\tht<\frac{1}{2}$.

Finally, we need to obtain the estimates on the remainder term:
\begin{align*}
    \int_\bbR \frac{(\rd_\xi \zt(\xi) - \rd_\xi \zt(\eta))(\calD^n(\xi,\xi)-\calD^n(\xi,\eta))}{\xi-\eta} \ud\eta.
\end{align*}
Observe that it has a similar structure to the right hand side of \eqref{eq:Hfg formula}. That is, we can extract a decay of order $\vphi^2$ from the differences of $\rd_\xi \zt$ and $\calD^n$. From this, we obtain that 
\begin{align*}
    \left\Vert\int_\bbR \frac{(\rd_\xi \zt(\xi) - \rd_\xi \zt(\eta))(\calD^n(\xi,\xi)-\calD^n(\xi,\eta))}{\xi-\eta} \ud\eta \right\Vert_{C^\vphi} \aleq \nrm{\rd_\xi \zt}_{C^\vphi} \nrm{\calD^n}_{C^\vphi}.
\end{align*}

Combining all these estimates, we conclude that
\begin{align*}
    P.V. \int_\bbR \sum_{n=0}^\infty \frac{\rd_\xi \zt(\eta)}{\xi-\eta} \calD^n(\xi,\eta) \ud \eta = \sum_{n=0}^\infty P.V. \int_\bbR \frac{\rd_\xi \zt(\eta)}{\xi-\eta} \calD^n(\xi,\eta) \ud \eta \in C^{\vphi}.
\end{align*}

\noindent\textbf{$\bullet$ $C^{1,\vphi}$ Estimates on $\calH[v[\gma]]$}

Now we present the way to establish the $C^{1,\vphi}$ bound of $\calH[v[\gma]]$. Since the regular terms are at least $C^{1,\alp}$ with enough decay, they are at least $C^{1,\alp}$ for all $0<\alp<1$. Therefore we focus on the singular term $\calV[\zt]$. For simplicity, we introduce the way to estimate $\calH[S]$. In the above step, we establish that 
\begin{align*}
    \calH[S[\zt]] = \sum_{n=0}^\infty (-1)^n S_{n}[\zt].
\end{align*}
Due to the linearity of Hilbert transform, it suffices to show that the $C^{\vphi}$ norms of $\calH[S_1]$ are summable. Recall that 
\begin{align*}
    \calH[S_n](\xi) = P.V.\, \int_\bbR \frac{1}{\xi-\eta} P.V.\, \int_\bbR \frac{\rd_\xi \zt_1(\tht)}{\eta-\tht} \calD^n (\eta,\tht) \ud \tht \ud \eta.
\end{align*}
To the best of our knowledge, there is no way to directly obtain the $C^{1,\vphi}$ bound rather than a $C^{1,\tld{\vphi}}$ bound of $\calH[S_n]$. Hence we proceed to obtain the additional structural properties. After expanding $\calD^n$, we observe that $\calH[S_{1,n}]$ is a sum of the operators of $\calT_j^{n_1,n_2}[\zt]$, which is defined as 
\begin{align*}
    \calT_0^{n_1,n_2}[\zt](\xi) := P.V.\, \int_\bbR \frac{1}{\xi-\eta} \calQ[\zt_1]^{n_1} (\xi,\eta) \calQ[\zt_2]^{n_2} (\xi,\eta) \ud \eta, \\
    \calT_j^{n_1,n_2}[\zt](\xi) := P.V.\, \int_\bbR \frac{\rd_\xi \zt_j(\eta)}{\xi-\eta} \calQ[\zt_1]^{n_1} (\xi,\eta) \calQ[\zt_2]^{n_2} (\xi,\eta) \ud \eta.
\end{align*}
In contrast to $\calH[S_n]$, $\calH[\calT_0^{n_1,n_2}[\zt]]$ exhibits an additional structure. That is, we can find the exact formula of it by using the the algebraic structure of difference quotients 
\begin{align} \label{eq:usefulformula2}
    \frac{\calQ[\zt_1]^{n_1} \calQ[\zt_2]^{n_2} (\xi,\eta)}{\xi-\eta}  
    & = \frac{n_1}{n} \frac{\rd_\xi \zt_1(\xi)}{\xi-\eta} 
    + \frac{n_2}{n} \frac{\rd_\xi \zt_2(\xi)}{\xi-\eta}  
    + \frac{1}{n} \frac{\ud}{\ud \xi} \left[ \calQ[\zt_1]^{n_1}\calQ[\zt_2]^{n_2} \right] (\xi,\eta),
\end{align}
for $n=n_1+n_2$. By applying this formula to $\calH[\calT_j^{n_1,n_2}[\zt]]$ and integrating by parts, we establish the following formula:
\begin{align} \label{eq:T formula}
    \begin{split}
        \calH[\calT_j^{n_1,n_2}[\zt]] 
        & = \frac{n_1}{n} \calH[\rd_\xi \zt_1 \calT_1^{n_1-1,n_2}[\zt]] 
            + \frac{n_2}{n} \calH[\rd_\xi \zt_2 \calT_2^{n_1,n_2-1}[\zt]] \\
        & \quad - 2\rd_\xi \zt_j \left[ \frac{n_1}{n} \calH[\calT_1^{n_1-1,n_2}[\zt]] + \frac{n_2}{n} \calH[\calT_2^{n_1,n_2-1}[\zt]] \right]+ \calH[\rd_\xi \zt_j] \calT_0^{n_1,n_2}[\zt] \\
        & \quad + J_{\mathrm{rem}},
    \end{split}
\end{align}
where the remainder term $J_{\mathrm{rem}}$ is given by
    \begin{align*}
        J_{\mathrm{rem}}(\xi) & := \frac{1}{\pi} P.V.\int_\bbR \frac{\rd_\xi \zt_j(\xi) - \rd_\xi \zt_j(\tht)}{\xi-\tht} (\calN(\xi,\xi) - \calN(\xi,\tht)) \ud\tht, \\
        \calN(\xi,\tht) & := \calH_1[\calQ[\zt_1]^{n_1}\calQ[\zt_2]^{n_2}](\xi,\tht),
    \end{align*}
which allow us to calculate its norm inductively in $n$. Applying our various estimates on \eqref{eq:T formula} yields 
\begin{align*} 
    \nrm{\calH[\calT_j^{n_1,n_2}[\zt]]}_{C^\vphi} \aleq \vphi(\dlt)^{C(\tht)} (M\alp_1)^{n_1+n_2+1}.
\end{align*}
for some constants $C(\tht)$ and $M$ depending on the norms of $\zt$ and $\calH[\zt]$. Combining all these estimates, we can deduce that 
\begin{align*}
    \sum_{n=0}^\infty \nrm{\calH[S_n[\zt]]}_{C^\vphi} < \infty.
\end{align*}
The rest of the terms can be treated in a similar way. 

To complete the proof by using the fixed point theorem, we last need a Lipschitz estimates on the operators $v[\gma]$ and $\calH[v[\gma]]$. After applying careful telescoping sum formula on the differences of nonlinear term, these follow from additional calculations which are in the same spirit of above argument - see \ref{Subsec:Lip Estimates of v, Hv}.

\begin{remark} \phantom{=}
    \begin{enumerate}
        \item The above formulas \eqref{eq:Hfg formula}, \eqref{eq:usefulformula2} and \eqref{eq:T formula} provide a framework for dealing with singular integral operators in the critical modulus of continuity spaces.
        \item Our proof can also be applied to the well-posedness problem of vortex patch in $W^{2,1}$ or $H^\frac{3}{2}$, under suitable assumptions on the modulus of integrability.
        \item The operator $\calT_j^{n_1,\cdots,n_l}[f]$ also frequently appears in other free surface equations, such as the vortex sheet problem in the two-dimensional incompressible Euler equations, the Muskat problem and others. For example, in \cite{DuchonRobertVortexSheet}, Duchon and Robert analyze $\calT_j^{n_1,\cdot,n_l}[f]$ in terms of the Fourier transform to deal with the existence problem of vortex sheet. This paper provides a framework for dealing with such operators in critical-regularity spaces. 
    \end{enumerate}
\end{remark}

\subsection{Organization}
Our paper is organized as follows.
\begin{itemize}
    \item In Section \ref{Sec:Prelim}, we introduce our notations and present basic observations about the Hilbert transform in spaces $C^{\vphi}$.
    \item In Section \ref{Sec:LWP}, we prove the local well-posedness of the \eqref{eq:HCDE} in $C^{1,\vphi}$ and prove Theorem \ref{Thm2}
    \item In Section \ref{Sec:Proof of Lem}, we prove the various lemmas appearing in the previous section.
    \item In Section \ref{Sec:Discussions}, we discuss the global well-posedness
\end{itemize}

\section{Preliminaries} \label{Sec:Prelim}

\subsection{Notations}
We present notations used throughout the paper.
\begin{itemize}
    \item $\vphi$: A modulus of continuity satisfying Assumptions \ref{A1}--\ref{A3}
    \item $\tld{\vphi}$: The induced modulus of continuity of $\vphi$, defined by $\tld{\vphi}(h) = \int_0^h \frac{\vphi(r)}{r} \ud r.$
    \item $|\gma|_\star$: The infimum of arc-chord of a parametrization $\gma$, defined by 
    $|\gma|_\star := \inf_{\xi\neq\eta} \left| \frac{\gma(\xi)-\gma(\eta)}{\xi-\eta} \right|$.
    \item $\calQ[f](x,y)$: The difference quotient of a function $f$, given by $\calQ[f](x,y):= \frac{f(x)-f(y)}{x-y}$.
    \item $\calH$: The Hilbert transform. For a multi-variable function, $\calH_i$ denotes the transform with respect to $i$-th argument.
    \item $f_i$: The $i$-th component of a vector or vector-valued function $f$.
    \item $A\aleq B$: An inequality of the form $A\leq CB$, where $C$ is a constant depending only on $\vphi$.
\end{itemize}

\subsection{Hilbert Transform and Its Properties} 
In this subsection, we present several lemmas concerning the properties of the Hilbert transform, which will be used in Section \ref{Sec:LWP}. For the sake of readability, their proofs are deferred to Section \ref{Sec:Proof of Lem}. For a function $f : \bbR \to \bbR$ with sufficient decay, the Hilbert Transform $\calH[f]$ is defined by 
\begin{align*}
    \calH[f](x) = \frac{1}{\pi} P.V.\int_\bbR \frac{f(y)}{x-y} \ud y.
\end{align*}
If $f$ is defined on $\bbT$, then the Hilbert transform is 
\begin{align*}
    \calH[f](x) = \frac{1}{2\pi} P.V.\int_\bbT f(y) \cot\left( \frac{x-y}{2} \right) \ud y.
\end{align*}
As mentioned earlier, the Hilbert transform is unbounded on $C^\vphi$ for a general modulus $\vphi$. Nevertheless, we obtain the following mapping properties of $\calH :C^\vphi \to C^{\tld{\vphi}}$.
\begin{lemma} [$C^{\tld{\vphi}}$ estimate of {$\calH[f]$}] \label{Lem:img of H}
    Let $f \in C_c^\vphi(\bbR)$ with a compact support. Then its Hilbert transform $\calH[f]$ belongs to $C^{\tld{\vphi}}(\bbR)$. Furthermore, the following bounds hold:
    \begin{align} 
            & \nrm{\calH[f]}_{C^0} \aleq \nrm{f}_{C^0}^{1-\tht} \nrm{f}_{C^\vphi}^\tht,  \label{eq:Hf C^0} \\
            & \nrm{\calH[f]}_{C^{\tld{\vphi}}} \aleq \nrm{f}_{C^\vphi}.\notag
    \end{align}
\end{lemma}
We remark that the compact support assumption on $f$ can be relaxed to the $L^p$ decay assumption for $p>1$. Also, we note that the Hilbert transform of the compactly supported function decays as $\calH[f](x) = O(1/x)$ as $x \to \infty$. Due to the unboundedness of Hilbert transform on $C^0$, the norm $\nrm{\calH[f]}_{C^0}$ cannot be bounded solely by $\nrm{f}_{C^0}$.

We now examine the multiplication properties of the Hilbert transform on $C^\vphi$. Recall the well-known Cotlar identity \eqref{eq:Cotlar}:
\begin{align*} 
    \calH[fg] = f\calH[g] + \calH[f]g + \calH[\calH[f]\calH[g]].
\end{align*}
This identity holds for functions in $C^\vphi$ because our modulus $\vphi$ satisfies the Dini condition, which allows for a derivation via the Fourier transform. However, it does not provide direct information on the $C^\vphi$ regularity of $\calH[fg]$, due to the regularity loss in the last term. Nevertheless, we refine the Cotlar identity into the following product formula.

\begin{lemma} [An Algebra Property of the Hilbert Transform]\label{Lem:Hfg}
    Let $f, g \in C_c(\bbR)$ be continuous functions with compact support, and assume that their Hilbert transforms $\calH[f]$ and $\calH[g]$ belong to $C^\vphi(\bbR)$. Then the Hilbert transform of their product, $\calH[fg]$, is also in  $C^\vphi(\bbR)$. In particular, the following identity and corresponding norm estimates hold:
    \begin{align} \tag{\ref{eq:Hfg formula}}
        \calH[fg](x) = \frac{1}{\pi} P.V.\int_\bbR \frac{(\calH[f](x) - \calH[f](y)) (\calH[g](x) - \calH[g](y))}{x-y} \ud y,
    \end{align}
    with the norm estimates
    \begin{align} \tag{\ref{eq:Hfg nrm}}
        \nrm{\calH[fg]}_{C^\vphi}
        & \aleq \nrm{\calH[f]}_{C^\vphi} ( \nrm{\calH[g]}_{C^\vphi}^\tht \nrm{\calH[g]}_{C^0}^{1-\tht} + \nrm{\calH[g]}_{C^0} ) + \nrm{\calH[g]}_{C^\vphi} \left( \nrm{f}_{C^0} + \nrm{\calH[f]}_{C^0} \right),
    \end{align}
    or more simply,
    \begin{align} \label{eq:Hfg nrm_simple}
        \nrm{\calH[fg]}_{C^\vphi} \aleq \nrm{\calH[f]}_{C^\vphi} \nrm{\calH[g]}_{C^\vphi}.
    \end{align}
\end{lemma}
\begin{remark}
    As a corollary of the estimate of \eqref{eq:Hfg formula}, under the same assumptions on the functions, we have the following estimate:
    \begin{align} \label{eq:prod_cornrm}
        \left\Vert \int_\bbR \frac{(f(x)-f(y))(g(x)-g(y))}{x-y} \ud y \right\Vert_{C^\vphi} \aleq\nrm{f}_{C^\vphi} \nrm{g}_{C^\vphi},
    \end{align}
    This inequality will be repeatedly used in later sections.
\end{remark}
Note that in this lemma, we do not assume that $f$ or $g$ belongs to $C^\vphi$. This lemma and its proof form a cornerstone of our paper. It allows us to control the product of functions under the Hilbert transform without the regularity loss inherent in the classical Cotlar identity. To obtain the bound \eqref{eq:Hfg nrm}, the assumption \ref{A1} and the interpolation inequality \eqref{eq:interpolation} play a key role. Without these, we can only obtain \eqref{eq:Hfg nrm_simple}, which is not enough to establish the crucial estimates - Lemma \ref{Lem:Hf^n}, \ref{Lem:T2}    and \ref{Lem:T6}.

\begin{lemma} [Estimate of {$\calH[f^n]$}] \label{Lem:Hf^n}
    Suppose that a function $f \in C^\vphi$ is compactly supported in $[-\dlt,\dlt]$, and that its Hilbert transform $\calH[f]$ also belongs to $C^\vphi$. We further assume that the following bounds hold for some constant $M>0$:
    \begin{align*} 
        \nrm{f}_{C^0} \aleq M \vphi(\dlt), \quad \nrm{f}_{C^\vphi} \aleq M, \quad
        \nrm{\calH[f]}_{C^0} \aleq M \vphi(\dlt)^{1-\tht}, \quad \nrm{\calH[f]}_{C^\vphi} \aleq M \vphi(\dlt)^{-\tht}.
    \end{align*}
    Let $C_{P,1}$ be a constant depending only on $\vphi$. If $\dlt$ is sufficiently small such that $\vphi(\dlt)^\tht < C_{P,1}/2$, then for $\alpha_1 := C_{P,1}\vphi(\dlt)^{1-2\tht}$, we have the following exponential bound:
    \begin{align} \label{eq:Hf^n nrm}
        \nrm{\calH[f^n]}_{C^\vphi} \aleq \vphi(\dlt)^{\tht-1} (M\alpha_1)^n.
    \end{align}
\end{lemma}
\begin{remark} \label{Rmk:prefactor}
    The apparent divergence of the prefactor $\vphi(\dlt)^{\tht-1}$ as $\dlt \to 0$ is not an issue in our argument. Since the proof only requires satisfying a finite set of smallness conditions on $\dlt$, we can fix a single positive $\dlt$ throughout the analysis. Consequently, this prefactor becomes a constant, and the convergence of the series is driven by the term $(M\alp_1)^n$.
\end{remark}
We note that the assumptions on $f$ are tailored to the norm of the derivative of the localized function $\rd_\xi \zt$. Also, as discussed earlier, since $\alpha_1 = C_{P,1} \vphi(\dlt)^{1-2\tht}$, to ensure the exponent of $\vphi(\dlt)$ is non-negative, we need to assume $0<\tht<\frac{1}{2}$.

Finally, we recall a fundamental property that the Hilbert transform commutes with the difference quotient (and thus with differentiation):
\begin{align*}
    \calH[\calQ [f]] = \calQ[\calH[f]].
\end{align*}

\section{Local Well-Posedness of Vortex Patches} \label{Sec:LWP}

In this section we prove Theorems \ref{Thm1} and \ref{Thm2}. As we have discussed in the section \ref{Subsec:Main Results}, we consider equations for $\gma$ and $\calH[\gma]$ simultaneously: 
\begin{align} \tag{HCDE}
    \begin{cases}
        & \rd_t \gma = v[\gma] \\
        & \rd_t \calH(\gma) = \calH[v[\gma]] \\
        & \gma_{t=0} = \gma_0 \\
        & \calH[\gma]_{t=0} = \calH[\gma_0].
    \end{cases}
\end{align}
The proof is based on the standard fixed-point argument. 

\subsection{Functional Setup}
In this subsection, we define the functional framework for the fixed-point argument. We consider the Banach space for the pair $(\gma, \calH[\gma])$ defined as $\bfB_\vphi = (C^{1,\vphi}(\bbT;\bbR^2))^2$. To guarantee a $C^1$ parametrization to be regular parametrization, we need to control the arc-chord
\begin{align*}
    |\gma|_\star = \inf_{\xi \neq \eta} \frac{|\gma(\xi)-\gma(\eta)|}{|\xi-\eta|}.
\end{align*}
For a parametrization $\gma \in \bfB_\vphi$, the condition $|\gma|_{\star} > 0$ ensure that $\gma$ has no self-intersection. Moreover, by the inverse function theorem, this condition ensures that an actual curve and its parametrization $\gma$ have the equal regularity. We therefore define an open set $\bfO^M$ as
\begin{align*}
    \bfO^M = \left\{(\gma^1,\gma^2) \in \bfB_\vphi : \nrm{\gma^1}_{C^{1,\vphi}}, \nrm{\gma^2}_{C^{1,\vphi}} < M, |\gma^1|_\star > \frac{1}{M} \right\}.
\end{align*} 
Note that no assumption is made on the arc-chord norm of $\gma^2$, which represents $\calH[\gma]$.

\subsection{Regularity of the Velocity Field and Its Hilbert Transform} \label{Subsec:Estimates of v, Hv}
In this subsection, we prove that the velocity operator $v$ and its Hilbert transform $\calH[v]$ is indeed $(C^{1,\vphi}(\bbT))^2$. We first present the result for $v$.
\begin{proposition} \label{Prop:est of v}
    Let $(\gma,\calH(\gma)) \in \bfO^M$ for some $M>1$. Then the velocity operator $v = v[\gma]$ satisfies $C^{1,\vphi}(\bbT;\bbR^2)$. In particular, the derivative of $v$ satisfies
    \begin{align} \label{eq:dv formula}
        \rd_\xi v[\gma](\xi) = P.V.\int_\bbT \frac{(\gma(\xi)-\gma(\eta))\cdot\gma(\xi)}{|\gma(\xi) - \gma(\eta)|^2} \rd_\xi \gma(\eta) \ud \eta.
    \end{align}
    In addition, for any $M>1$ there exists a constant $C(M)$ such that
    \begin{align*} 
        \nrm{v[\gma]}_{C^{1,\vphi}} \aleq C(M).
    \end{align*}
\end{proposition}
\begin{proof}[Proof of Proposition \ref{Prop:est of v}]
    Before we begin, we observe that it suffices to prove the following:
    \begin{align*} 
        P.V.\int_\bbT \frac{(\gma(\xi)-\gma(\eta))\cdot\gma(\xi)}{|\gma(\xi) - \gma(\eta)|^2} \rd_\xi \gma(\eta) \ud \eta \in C^\vphi.
    \end{align*}
    Once this is established, the conclusion that this representation is equivalent to $\rd_\xi v[\gma](\xi)$ follows directly from the same argument in the proof of lemma 6 in the chapter 8.3 of \cite{MB02}. To simplify the notation, we abuse it and denote the right-hand side of \eqref{eq:dv formula} as $\rd_\xi v[\gma]$. Let us begin by outlining the steps of our proof. In step 1, we will clearly state what we aim to prove and discuss the decomposition of $\rd_\xi v[\gma]$. In Step 2, we examine the norm of the localized function $\zt$ of the given curve $\gma$. In step 3 and 4, we derive the estimates on the regular and singular part, respectively. In step 5, we present the way to obtain the summability on the norm of operators appearing in step 3. 

    \noindent\textbf{STEP 1: Set up and decomposition of $\rd_\xi v[\gma]$} 

    We derive the pointwise estimates for the $C^\vphi$ norm of $\rd_\xi v[\gma]$ through this proof. i.e. we show that for each $\tht \in \bbT$, 
    \begin{align} \label{eq:ptwise nrm}
        \sup_{|\xi-\tht|<\frac{1}{4}\delta} \frac{|\rd_\xi v[\gma](\xi) - \rd_\xi v[\gma](\tht)|}{|\xi-\tht|} \leq C
    \end{align}
    for small $\delta = \delta(\vphi,M)$, which will be defined later. With out loss of generality, we can assume that $\tht = 0$. In addition, the difference of $\rd_\xi v$ is invariant under the translation of the curve $\gma$, thus we also assume $\gma(0) = (0,0)$. We now proceed to decompose $\rd_\xi v[\gma]$ in the singular and regular part by separating the range of integral into the two region: one where $\xi$ and $\eta$ is close enough, and one where they are far apart. To accomplish this, take the bump function $\chi \in C^\infty(\bbT)$ defined by 
    \begin{align*}
        \chi (\xi) = \begin{cases}
            1  & \quad \xi \in [-\frac{1}{8}\delta,\frac{1}{8}\delta] \\
            0  & \quad \xi \in [-\frac{1}{4}\delta,\frac{1}{4}\delta]^c.
        \end{cases}
    \end{align*}
    Using this, we can decompose $\rd_\xi v[\gma](\xi)$ as follows:
    \begin{align*}
        \begin{split}
            \rd_\xi v[\gma](\xi)
            & = P.V.\int_\bbT \frac{(\gma(\xi)-\gma(\eta))\cdot\gma(\xi)}{|\gma(\xi)-\gma(\eta)|^2} \rd_\xi \gma(\eta) \ud \eta \\
            & = \underbrace{P.V.\int_{|\xi-\eta|<\delta/4} \frac{(\gma(\xi)-\gma(\eta))\cdot\gma(\xi)}{|\gma(\xi)-\gma(\eta)|^2} \rd_\xi \gma(\eta) \chi(\xi-\eta) \ud \eta}_{=:\rd_\xi v_s[\gma]} \\
            & \quad + \underbrace{\int_{|\xi-\eta|>\delta/8} \frac{(\gma(\xi)-\gma(\eta))\cdot\gma(\xi)}{|\gma(\xi)-\gma(\eta)|^2} \rd_\xi \gma(\eta) (1-\chi(\xi-\eta)) \ud \eta}_{=:\rd_\xi v_s[\gma]}.
        \end{split}
    \end{align*} 
    The subscript $r$ and $s$ represent the regular and singular part of $\rd_\xi v$ respectively. To deal with the singular term $\rd_\xi v_s[\gma]$, we introduce the linear approximation of $\gma$ around the $\xi = 0$. Define a new function $\zt$ as follows: 
    \begin{align*}
        \begin{cases}
            & \zt(\xi) = (\gma(\xi) - U \xi) \tld{\chi}(\xi) \\
            & U = \rd_\xi \gma(0) \\
            & |U| \neq 0.
        \end{cases}
    \end{align*} 
    where $\tld{\chi}(\xi) := \chi(\xi/4)$ is a fattend cutoff function. $\zt_i$ is well-defined on $\bbT$ and satisfies $\rd_\xi \zt_i (0) = 0$. Recall that the integral range of the first term $I_1$ is $|\xi-\eta| < \frac{\delta}{4}$ and we restrict ourselves to the case $|\xi| < \frac{\delta}{4}$ - see \eqref{eq:ptwise nrm}. Thus, we have $|\xi|, |\eta| < \delta/2 $. Given this, we can substitute $\zt(\xi) + U\xi\tld{\chi}(\xi)$ for $\gma$ in the singular term. For notational convenience, we define $\tld{\xi} = \xi\tld{\chi}(\xi)$ and we recast the $\rd_\xi v_s[\gma]$ as follows:
    \begin{align*}
        \begin{split}
            \rd_\xi v_s[\gma] (\xi) 
            & = P.V.\int_{|\xi-\eta|<\delta/4} \frac{(\gma(\xi)-\gma(\eta))\cdot\gma(\xi)}{|\gma(\xi)-\gma(\eta)|^2} \rd_\xi \gma(\eta) \chi(\xi-\eta) \ud \eta \\
            & = P.V.\int_{|\xi-\eta|<\delta/4} \frac{[(\zt(\xi)-\zt(\eta))+U(\tld{\xi}-\tld{\eta})]\cdot(\zt(\xi) + U\tld{\xi})} {|(\zt(\xi)-\zt(\eta)) + U(\tld{\xi}-\tld{\eta})|^2} (\rd_\xi \zt(\eta)+ U) \chi(\xi-\eta) \ud \eta \\
            & = P.V.\int_{|\xi-\eta|<\delta/4} \frac{\chi(\xi-\eta)}{\xi-\eta} \frac{[\calQ[\zt](\xi,\eta) + U]\cdot(\zt(\xi)+ U\tld{\xi})}{|\calQ[\zt](\xi,\eta)+U|^2} (\rd_\xi \zt(\eta)+ U) \ud \eta,
        \end{split}
    \end{align*}
    where $\calQ[\zt](\xi,\eta)$ denote the difference quotient of $\zt$. In the last equality, we use that in the regime $|\xi-\eta|<\delta/4$ and $|\xi|<\delta/4$, we have $\chi(\xi-\eta) = 1$. 
    
    The treatment of the difference quotient on the torus is somewhat delicate due to the singular factor $1/(x-y)$. To avoid this difficulty, we extend the above expression from $\bbT$ to the whole space $\bbR$.\footnote{This procedure is needed for notational and calculational convenience. One may prove the whole procedure regarding $\zt$ as a function on $\bbT$.}
    From now on, We extend $\zt, \chi$ on $\bbR$ by defining $\zt(\xi) = \chi(\xi) = 0$ for $|\xi| > \pi$\ and consider them as vector-valued functions on $\bbR$. Define the operator $\calV$ of $\zt$ as
    \begin{align*}
        & \calV[\zt] : \bbR \to \bbR^2 \\
        & \calV[\zt](\xi) :=  P.V.\int_\bbR \frac{\chi(\xi-\eta)}{\xi-\eta} \frac{[\calQ[\zt](\xi,\eta) + U]\cdot(\zt(\xi)+ U\tld{\xi})}{|\calQ[\zt](\xi,\eta)+U|^2} (\rd_\xi \zt(\eta)+ U) \ud \eta.
    \end{align*}
    We remark that $\rd_\xi v_s[\gma](\xi) = \calV[\zt](\xi)$ for $|\xi|<\dlt/4$.

    \noindent\textbf{STEP 2: Norm of the localized function $\zt$}

    In this step, we consider the relation between the norms of the localized function $\zt$, $\calH[\zt]$ and the localizing spatial scale $\delta$. First, we examine the scaling properties of $C^{1,\vphi}$. For a smooth function $f$, denote the re-scaled function as $f_\dlt(x) = f(x/\dlt)$. Then one can obtain the bounds 
    \begin{align*}
        \begin{cases}
            \nrm{f_\dlt}_{C^0} \leq \nrm{f}_{C^0} \\
            \nrm{f_\dlt}_{C^1} \leq \frac{1}{\dlt} \nrm{f}_{C^1}.
        \end{cases}
    \end{align*} 
    For the $C^{1,\vphi}$ scaling property, we need the assumption \ref{A2}: 
    \begin{align} \label{eq:scaling of C^1,vphi}
        \begin{split}
            \nrm{f_\dlt}_{\dot{C}^{1,\vphi}} = \frac{1}{\dlt}\nrm{f_\dlt'}_{\dot{C}^\vphi} 
            & = \frac{1}{\dlt}\sup_{x\neq y} \frac{|f_\dlt'(x)-f_\dlt'(y)|}{\vphi(|x-y|)} \\  
            & = \frac{1}{\dlt}\sup_{x\neq y} \frac{|f'(\dlt^{-1}x)-f'(\dlt^{-1}y)|}{\vphi(|\dlt^{-1}(x-y)|)} \smash{\underbrace{\frac{\vphi(\dlt^{-1}|x-y|)}{\vphi(|x-y|)}}_{\aleq \frac{1}{\vphi(\dlt)}}} \\
            & \aleq \nrm{f}_{\dot{C}^{1,\vphi}} \frac{1}{\dlt\vphi(\dlt)}.
        \end{split}
    \end{align}
    Using this scaling property, we obtain the norms of $\chi$ with respect to $\dlt$: $\nrm{\chi}_{\dot{C}^{1,\vphi}} \aleq (\dlt\vphi(\dlt))^{-1}$. Using the algebra properties of $C^{1,\vphi}$, we also obtain the norm of $\zt$:  
    \begin{align*}
        \begin{cases}
            & \nrm{\zt}_{C^0} \leq M \dlt \vphi(\dlt) \\
            & \nrm{\zt}_{C^1} \leq \frac{C}{\dlt} M \dlt \vphi(\dlt) + M \vphi(\dlt) \aleq M\vphi(\dlt) \\
            & \nrm{\zt}_{C^{1,\vphi}} \aleq M.
        \end{cases}
    \end{align*}
    Now we examine the norms of $\calH[\zt]$. Since we cannot guarantee that $\calH[\gamma](0) = (0,0)$, we only have the bound $\nrm{\calH[\zt]}_{C^0} \aleq M$. For the $C^1$ norm, we use \eqref{eq:Hf C^0} on $\zt$: 
    \begin{align*}
        \nrm{\calH[\zt]}_{C^1} \aleq \nrm{\zt}_{C^1}^{1-\tht} \nrm{\zt}_{C^{1,\vphi}}^\tht \aleq M \vphi(\dlt)^{1-\tht}.
    \end{align*}
    Now we estimate the $C^{1,\vphi}$ norm of $\calH[\zt]$. Denote $\bar{\gma} = \gma-U\xi$. From the fact that $\calH[\rd_\xi \bar{\gma}](0) = 0$, we already have that $\nrm{\calH[\bar{\gma}]}_{C^1([-\dlt,\dlt])} \leq M\vphi(\dlt)$. We decompose the estimate as 
    \begin{align*}
        \nrm{\calH[\zt]}_{C^{1,\vphi}} & \leq \nrm{\calH[\rd_\xi \bar{\gma} \chi]}_{C^\vphi} + \nrm{\calH[\bar{\gma} \rd_\xi \chi]}_{C^\vphi}.
    \end{align*}
    We apply the product estimate in \eqref{eq:Hfg nrm} to each term. For the first term, we have 
    \begin{align*}
        \begin{split}
            \nrm{\calH[\rd_\xi \bar{\gma} \chi]}_{\dot{C}^\vphi}
            & \aleq \nrm{\calH[\chi]}_{\dot{C}^\vphi} ( \nrm{\calH[\rd_\xi \bar{\gma}]}_{C^0}^{1-\tht}\nrm{\calH[\rd_\xi \bar{\gma}]}_{\dot{C}^\vphi}^\tht + \nrm{\calH[\rd_\xi \bar{\gma}]}_{C^0} ) \\
            & \quad + \nrm{\calH[\rd_\xi \bar{\gma}]}_{\dot{C}^\vphi} (\nrm{\calH[\chi]}_{C^0} + \nrm{\chi}_{C^0}) \\
            & \aleq \frac{1}{\vphi(\dlt)} M \vphi(\dlt)^{1-\tht} + M \\
            & \aleq M\vphi(\dlt)^{-\tht}.
        \end{split}
    \end{align*}
    Likewise, for the second term we have that 
    \begin{align*}
        \begin{split}
            \nrm{\calH[\bar{\gma} \rd_\xi \chi]}_{\dot{C}^\vphi} 
            & \aleq \nrm{\calH[\bar{\gma}]}_{\dot{C}^\vphi} ( \nrm{\calH[\rd_\xi \chi]}_{C^0}^{1-\tht} \nrm{\calH[\rd_\xi \chi]}_{\dot{C}^\vphi}^\tht  + \nrm{\calH[\rd_\xi \chi]}_{C^0} ) \\
            & \quad + \nrm{\calH[\rd_\xi \chi]}_{\dot{C}^\vphi} (\nrm{\bar{\gma}}_{C^0} + \nrm{\calH[\bar{\gma}]}_{C^0}) \\
            & \aleq M \dlt \left( \frac{1}{\dlt^{1-\tht}} \frac{1}{(\dlt\vphi(\dlt))^\tht} + \frac{1}{\dlt} \right) + 
            \frac{M}{\dlt \vphi(\dlt)} \dlt \vphi(\dlt) \\
            & \aleq M\vphi(\dlt)^{-\tht}.
        \end{split}
    \end{align*}
    Hence, we conclude that 
    \begin{align*}
        \nrm{\calH[\zt]}_{C^{1,\vphi}} \aleq M\vphi(\dlt)^{-\tht}.
    \end{align*}

    \noindent\textbf{STEP 3: Estimates of the regular part} 

    The integral for $\rd_\xi v_r$ is taken over the range $|\xi-\eta| > \delta/8$. In this regime, the integrand is bounded and differentiable with respect to $\xi$. Consequently, $\rd_\xi v_{r}$ satisfies $C^1$ and we can compute its norm.
    \begin{align*}
        \nrm{\rd_\xi v_r}_{C^1} \aleq C(M) \dlt^{-1}.
    \end{align*}
    Moreover, by applying the embedding, we obtain that 
    \begin{align} \label{eq:regular part estimate}
        \nrm{\rd_\xi v_r}_{C^\alp} \aleq C(M,\alp) \dlt^{-1} \quad \forall 0 < \alpha < 1.
    \end{align}

    \noindent\textbf{STEP 4: Estimates of the singular part} 

    Since $\rd_\xi v_s[\gma](\xi) = \calV[\zt](\xi)$ for $|\xi|<\dlt/4$, we here present the $C^{\vphi}$ estimate of $\calV[\zt]$. We begin by decomposing it into the two parts:
    \begin{align*}
        \calV[\zt] (\xi) 
        & = P.V.\int_\bbR \frac{1}{\xi-\eta} \frac{[\calQ[\zt](\xi,\eta) + U]\cdot(\zt(\xi)+ U\tld{\xi})}{|\calQ[\zt](\xi,\eta)+U|^2} (\rd_\xi \zt(\eta)+ U) \ud \eta \\
        & \quad + P.V.\int_\bbR \frac{1-\chi(\xi-\eta)}{\xi-\eta} \frac{[\calQ[\zt](\xi,\eta) + U]\cdot(\zt(\xi)+ U\tld{\xi})}{|\calQ[\zt](\xi,\eta)+U|^2} (\rd_\xi \zt(\eta)+ U) \ud \eta \\
        & = \calV_1[\zt](\xi) + \calV_2[\zt](\xi).
    \end{align*}
    Note that $\calV_2[\zt]$ is defined as a principal value since there are some constant terms in the integrand. However, after expanding it, one readily obtains the $C^{\alp}$ bound for $\calV_2[\zt]$ as in a similar manner used in Step 3:
    \begin{align*}
        \nrm{\rd_\xi v_{s,2}}_{C^\alp} \aleq C(M,\alp) \dlt^{-1}.
    \end{align*}
    Also, from the compact support condition of $\zt$, we obtain the decay of $\calV_2[\zt]$:
    \begin{align*}
        |\calV_2[\zt](\xi)| \aleq \frac{1}{|\xi|} \quad \text{as $|\xi|\to \infty$}.
    \end{align*}
    We now consider the estimates of $\calV_1[\zt]$. Taking $\delta$ small enough such that $\nrm{\calQ[\zt]}_{C^0} \leq |\zt|_{C^1} \aleq M\vphi(\dlt) < 1$, we apply a power series expansion on the denominator. 
    \begin{align*}
        \begin{split}
            \frac{1}{|\calQ[\zt](\xi,\eta)+U|^2} 
            & = |U|^{-2} \left( 1+\frac{2}{|U|}(U\cdot \calQ[\zt](\xi,\eta)) + \frac{1}{|U|^2}|\calQ[\zt](\xi,\eta)|^2 \right) ^{-1} \\
            & = |U|^{-2} \sum_{n=0}^\infty (-1)^n \calD(\xi,\eta)^n,
        \end{split}
    \end{align*}
    where
    \begin{align*}
        \calD(\xi,\eta) = \frac{2}{|U|}(U\cdot \calQ[\zt](\xi,\eta)) + \frac{1}{|U|^2}|\calQ[\zt](\xi,\eta)|^2.
    \end{align*}
    We now obtain the following expression for $\calV_1[\zt]$:
    \begin{align*}
        \calV_1[\zt](\xi)
        & = \frac{1}{|U|^2} P.V.\int_\bbR \sum_{n=0}^\infty (-1)^n \frac{1}{\xi-\eta} [\calQ[\zt](\xi,\eta)+U]\cdot(\zt(\xi)+U\tld{\xi}) \calD^n(\xi,\eta) (\rd_\xi \zt(\xi)+U) \ud \eta.
    \end{align*}
    Observe that $\calV_1[\zt]$ is of the form
    \begin{align} \label{eq:decomposition of v_s}
        \calV_1[\zt] = \sum_{i=1}^5 P_i(U,\tld{\xi},\zt(\xi)) S_i, 
    \end{align}
    where $P_i(U,\tld{\xi},\zt(\xi))$ is a polynomial of $\tld{\xi}$ and $\zt(\xi)$ with coefficients depending on $U$ and
    \begin{align} \label{eq:def of S_i}
        \begin{split}
            & S_1 : = P.V.\int_\bbR \sum_{n=0}^{\infty} (-1)^n \frac{1}{\xi-\eta} \calD^n(\xi,\eta) \ud \eta \\
            & S_2 : = P.V.\int_\bbR \sum_{n=0}^{\infty} (-1)^n \frac{1}{\xi-\eta} \calD^n(\xi,\eta) \rd_\xi \zt(\eta) \ud \eta \\
            & S_3 : = P.V.\int_\bbR \sum_{n=0}^{\infty} (-1)^n \frac{1}{\xi-\eta} \calD^n(\xi,\eta) \calQ[\zt](\xi,\eta) \ud \eta \\
            & S_4 : = P.V.\int_\bbR \sum_{n=0}^{\infty} (-1)^n \frac{1}{\xi-\eta} \calD^n(\xi,\eta) \calQ[\zt_1](\xi,\eta) \rd_\xi \zt (\eta) \ud \eta \\
            & S_5 : = P.V.\int_\bbR \sum_{n=0}^{\infty} (-1)^n \frac{1}{\xi-\eta} \calD^n(\xi,\eta) \calQ[\zt_2](\xi,\eta) \rd_\xi \zt (\eta) \ud \eta.
        \end{split}
    \end{align}
    Define $S_{i,n}$ as the principal value of $n$-th summand of $S_i$. For example, 
    \begin{align*}
        S_{1,n} := (-1)^n P.V.\int_\bbT \frac{1}{\xi-\eta} \calD^n(\xi,\eta) \ud \eta.
    \end{align*}
    The principal value and the infinite sum in $S_i$ can be interchanged provided that
    \begin{align} \label{eq:S_i summability}
        \sum_{n=0}^\infty \nrm{S_{i,n}}_{C^\vphi} < \infty,
    \end{align}
    which implies that each $S_i$ belongs to $C^\vphi$. In conclusion, it suffices to prove \eqref{eq:S_i summability} to finish the proof.

    \noindent\textbf{STEP 5: Proof of \eqref{eq:S_i summability}} 

    We begin by introducing the following useful lemma that provides a way to deal with $S_{i,n}$.
    \begin{lemma} [Estimates of {$\calA[f,g]$}] \label{Lem:Afg}
        Suppose that the functions $f \in C^\vphi(\bbR)$ and $g \in C^\vphi(\bbR^2)$ are compactly supported. We further assume that their Hilbert transforms $\calH[f]$ and $\calH_2[g]$ are also in $C^\vphi$, where $\calH_2$ denotes the Hilbert transform acting on the second argument. Define the operator $\calA[f,g]$ as follows:
        \begin{align*} 
            \calA[f,g](x) = \frac{1}{\pi} P.V.\int_\bbR \frac{f(y)}{x-y} g(x,y) \ud y. 
        \end{align*}
        Then $\calA[f,g]$ belongs to $C^\vphi(\bbR)$ and satisfies the norm estimate:
        \begin{align} \label{eq:Afg nrm}
        \nrm{\calA[f,g]}_{C^\vphi} 
        & \aleq \nrm{f}_{C^\vphi} \nrm{\calH_2[g]}_{C^\vphi} + \nrm{\calH[f]}_{C^\vphi} \nrm{g}_{C^\vphi} + \nrm{f}_{C^\vphi} \nrm{g}_{C^\vphi}.
        \end{align}
    \end{lemma}
    \noindent Observe that $S_{i,n}$ is of the form $\calA[f,g_n]$:
    \begin{align*}
        \begin{split}
            S_{i,n}
            & = P.V.\int_\bbR \frac{1}{\xi-\eta} f(\eta) g_n(\xi,\eta) \ud \eta \\
            & = \calA[f,g_n](\xi).
        \end{split}
    \end{align*}
    Here, $f = 1$ or $\rd_\xi \zt(\eta)$ and $g_n = \calD^n(\xi,\eta)$ or $\calQ[\zt_j](\xi,\eta) \calD^n(\xi,\eta)$. Using \eqref{eq:Afg nrm} and the norms of $\zt$, we obtain the following bounds on $S_{i,n}$:
    \begin{align*}
        \begin{split}
            \nrm{S_{i,n}}_{C^\vphi}
            & \aleq (\nrm{f}_{C^\vphi}+\nrm{\calH[f]}_{C^\vphi}) (\nrm{g}_{C^\vphi} + \nrm{\calH_2[g]}_{C^\vphi})\\
            & \aleq (M + M\vphi(\dlt)^{-\tht}) (1 + \nrm{\calQ[\zt]}_{C^\vphi}) (1 + \nrm{\calH_2[\calQ[\zt]]}_{C^\vphi}) (\nrm{D^n}_{C^\vphi} + \nrm{\calH_2[D^n]}_{C^\vphi} ) \\
            & \aleq C(M,\dlt) (\nrm{D^n}_{C^\vphi} + \nrm{\calH_2[D^n]}_{C^\vphi} ).
        \end{split}
    \end{align*}
    Hence, it suffices to show the summability of $\nrm{D^n}_{C^\vphi}$ and $\nrm{\calH_2[D^n]}_{C^\vphi}$. $\nrm{\calD^n}_{C^\vphi}$ can be controlled by the algebra properties of $C^\vphi$: 
    \begin{align} \label{eq:D^n nrm}
        \begin{split}
            \nrm{\calD}_{C^0} 
            & \leq 2\nrm{\calQ[\zt]}_{C^0} + \frac{1}{|U|^2}\nrm{\calQ[\zt]}_{C^0}^2 \aleq M\vphi(\dlt) + M^2 \vphi(\dlt)^2 \aleq M\vphi(\dlt) \\  
            \nrm{\calD^n}_{C^\vphi} 
            &\leq n\nrm{\calD}_{C^0}^{n-1} \nrm{\calD}_{C^\vphi} \aleq C(M) [M \vphi(\dlt)]^{n-1}.
        \end{split}
    \end{align}
    Thus, for small enough $\dlt$, we obtain the exponential decay on $\nrm{\calD^n}_{C^\vphi}$. To estimate $\nrm{\calH_2[\calD^n]}_{C^\vphi}$, we use Lemma \ref{Lem:Afg}. Since $\zt$ satisfies the assumption, for small enough $\dlt$, we obtain the exponential decay of $\nrm{\calH_2[\calD^n]}_{C^\vphi}$
    \begin{align*}
        \nrm{\calH[\calD^n]}_{C^\varphi} \aleq C(M) \alpha_1^n,
    \end{align*}
    with $\alpha_1=C_{P,1} \vphi(\dlt)^{1-2\tht}<1$ which completesa the proof.
\end{proof}

Now we present the $C^{1,\vphi}$ estimates for the Hilbert transform of the velocity operators $\calH[v[\gma]]$.
\begin{proposition} \label{Prop:est of Hv}
    Let $(\gma,\calH(\gma)) \in O^M$ for some $M>1$. Then the Hilbert transform of the velocity operator $\calH[v[\gma]]$ satisfies $C^{1,\vphi}(\bbR;\bbR^2)$.  
    In addition, for any $M>1$, there exists a constant $C(M)$ such that
    \begin{align*}
        \nrm{\calH[v[\gma]]}_{C^{1,\vphi}} \aleq C(M).
    \end{align*}
\end{proposition}
\begin{proof}[Proof of Proposition \ref{Prop:est of Hv}] 
    In this proof, we adopt the same notation and decomposition used in the proof of Proposition \ref{Prop:est of v}. We aim to show the pointwise estimate of the $C^\vphi$ norm. Since the Hilbert transform is linear, we can consider $\calH[\rd_\xi v_r]$ and $\calH[\rd_\xi v_s]$ separately.

    It has already been established that the regular part $\rd_\xi v_r$ satisfies $C^\alpha$ in \eqref{eq:regular part estimate}. From the continuity of the Hilbert transform on $C^\alpha$, we obtain
    \begin{align*}
        \calH[\rd_\xi v_r] \in C^\alpha.
    \end{align*}

    \noindent\textbf{$\bullet$ Decomposition of $\calH[\rd_\xi v_s[\gma]]$}

    We now present decomposition and localization procedure of $\calH[\rd_\xi v_s[\gma]]$. Recall that for a function $f: \bbT \to \bbR$, the Hilbert transform of $f$ is defined as
    \begin{align*}
        \calH[f](\xi) = P.V.\frac{1}{2\pi} \int_\bbT f(\eta) \cot\left( \frac{\xi-\eta}{2} \right) \ud \eta.
    \end{align*}
    Note that the kernel is obtained by periodizing the Hilbert transform on the real line,
    \begin{align*}
        \frac{1}{2\pi} \cot\left(\frac{x}{2}\right) = \frac{1}{\pi x} + \frac{1}{\pi} \sum_{n\geq 1} \frac{1}{x+2\pi n} + \frac{1}{x-2\pi n}.
    \end{align*}
    This idea motivates the following decomposition:
    \begin{align*}
        \calH[\rd_\xi v_s[\gma]] (\xi)
        & = \frac{1}{2\pi} P.V.\int_\bbT \cot\left( \frac{\xi-\eta}{2} \right) \rd_\xi v_s[\gma](\eta) \ud \eta \\
        & = \frac{1}{2\pi} \int_\bbT (1-\chi(2(\xi-\eta))) \cot\left(\frac{\xi-\eta}{2}\right) \rd_\xi v_s[\gma](\eta) \ud \eta \\
        & + \quad \frac{1}{2\pi} \int_\bbT \chi(2(\xi-\eta)) \left[ \cot\left(\frac{\xi-\eta}{2}\right) - \frac{2}{\xi-\eta} \right] \rd_\xi v_s[\gma](\eta) \ud \eta \\
        & + \quad \frac{1}{\pi} P.V.\int_\bbT \frac{\chi(2(\xi-\eta))}{\xi-\eta} \rd_\xi v_s[\gma](\eta) \ud \eta .
    \end{align*}
    Since the kernels of the first two terms are smooth and bounded, the terms themselves are smooth. We now restrict $|\xi|<\dlt/8$. Then the integral range of the last term is $\{|\xi-\eta|<\dlt/8\}$ and thus we can substitute $\calV[\zt]$ for $\rd_\xi v_s[\gma]$ and changing the integral range $\bbT$ to $\bbR$. 
    \begin{align*} 
        \frac{1}{\pi} P.V.\int_\bbT \frac{\chi(2(\xi-\eta))}{\xi-\eta} \rd_\xi v_s[\gma](\eta) \ud\eta
        & = \frac{1}{\pi} P.V.\int_\bbR \frac{\chi(2(\xi-\eta))}{\xi-\eta} \calV[\zt](\eta) \ud\eta \\
        & = \frac{1}{\pi} P.V.\int_\bbR \frac{1}{\xi-\eta} \calV_1[\zt](\eta) \ud\eta \\
        & \quad + \frac{1}{\pi} \int_\bbR \frac{\chi(2(\xi-\eta))-1}{\xi-\eta} \calV_1[\zt](\eta) \ud\eta \\
        & \quad + \frac{1}{\pi} \int_\bbR \frac{\chi(2(\xi-\eta))}{\xi-\eta} \calV_2[\zt](\eta) \ud\eta .
    \end{align*}
    The decay condition on $\calV_1[\zt]$ leads to the smoothness of the second term. Also, the decaing condition on $\calV_2[\zt]$ and the boundedness of $\calH$ in the H\"older space leads to the $C^\alp$ bound for the second term. We now consider the first term, which is clearly $\calH[\calV_1[\zt]]$. 

    Recall that in \eqref{eq:decomposition of v_s}, we established that $\calV_1[\zt]$ admits the following decomposition:
    \begin{align*}
        \calV_1[\zt](\xi) = \sum_{i=1}^5 P_i(U,\tld{\xi},\zt(\xi)) S_i.
    \end{align*}
    where $P_i(U,\tld{\xi},\zt(\xi))$ is a polynomial of $\tld{\xi}$ and $\zt(\xi)$ with coefficients depending on $U$. Since $P_i(U,\tld{\xi})$ is smooth, Lemma \ref{Lem:Hfg} implies that it suffices to verify that $\calH[S_i] \in C^\vphi$ for each $i$. We present the detailed calculation for $S_1$ and $S_4$. The remaining cases can be treated similarly.
    
    \noindent\textbf{$\bullet$ Estimates on $\calH[S_1]$} \\
    Let us proceed by analyzing the algebraic structure of $S_1$ in detail. Recall that $S_1$ is given by
    \begin{align*}
        S_1 
        & = \sum_{n=0}^\infty (-1)^n P.V.\int_\bbR \frac{1}{\xi-\eta} \calD^n(\xi,\eta)\ud \eta = \sum_{n=0}^\infty S_{1,n},
    \end{align*}
    where
    \begin{align*}
        \calD^n(\xi,\eta) & = \frac{2}{|U|} (U\cdot \calQ[\zt](\xi,\eta)) + \frac{1}{|U|^2} |\calQ[\zt]|^2(\xi,\eta).
    \end{align*}
    We shall decompose $\calD^n$ into the product of $\calQ[\zt_1]$ and $\calQ[\zt_2]$. For computational convenience, denote $\calD(\xi,\eta)$ and its power by
    \begin{align*}
        \begin{split}
            \calD (\xi,\eta)
            & = c_1\calQ[\zt_1] + c_2\calQ[\zt_2] + c_{11}\calQ[\zt_1]^2 + c_{12}\calQ[\zt_1] \calQ[\zt_2] + c_{22}\calQ[\zt_2]^2 \\
            \calD^n (\xi,\eta) & = \sum_{n_1,n_2} C_{n_1,n_2}^n \calQ[\zt_1]^{n_1} \calQ[\zt_2]^{n_2},
        \end{split}
    \end{align*}
    for some constants $c_1,\cdots,c_{22}$ and the combinatorial constants $C_{n_1,n_2}^n$. With the above representation, $S_{1,n}$ can be expressed as follows:  
    \begin{align*}
        S_{1,n} = (-1)^n \sum_{n_1,n_2} C_{n_1,n_2}^n P.V.\int_\bbR \frac{1}{\xi-\eta} \left( \calQ[\zt_1]^{n_1} \calQ[\zt_2]^{n_2} \right)(\xi,\eta) \ud \eta.
    \end{align*}
    In conclusion, dealing with $\calH[S_i]$ is reduced to considering the operators
    \begin{align} \label{eq:prop2_obj operator}
        \calH\left[ P.V.\int_\bbR \frac{1}{\xi-\eta} (\calQ[\zt_1]^{n_1} \calQ[\zt_2]^{n_2})(\xi,\eta) \ud \eta \right].
    \end{align}
    We now introduce a lemma ensuring the regularity of the above operators.
    \begin{lemma} [Formulas of {$\calH[\calT[f]]$}] \label{Lem:T}
        Let $f : \bbR \to \bbR^l$ be a vector-valued function. Define the operator $\calT_j^{n_1,\cdots,n_l}[f] : \bbR \to \bbR$ by
        \begin{align} \label{eq:def T}
            \calT_j^{n_1,\cdots,n_l}[f](x) = \begin{cases}
               \displaystyle P.V.\int_\bbR \frac{1}{x-y} \Pi \calQ[f]^{n_1,\cdots,n_l} (x,y) \ud x \quad j=0\\[10pt]
                \displaystyle P.V.\int_\bbR \frac{f'_j(x)}{x-y} \Pi \calQ[f]^{n_1,\cdots,n_l} (x,y) \ud y \quad j\in\{1,2,\cdots,l\},
                \end{cases}
        \end{align}
        where
        \begin{align} \label{eq:def PiQ}
            \Pi\calQ [f]^{n_1,\cdots,n_l} := \prod_{k=1}^l \calQ[f_k]^{n_k}.
        \end{align}
        
        Suppose that $f, \calH[f] \in C^{1,\vphi}(\bbR;\bbR^l)$ and $f$ has a compact support. Then $\calH[\calT_j^{n_1,\cdots,n_l}[f]]$ belongs to $C^\vphi$, and the following formulas hold.

        For $j=0$ and $n_1,\cdots,n_l \geq 1$,
        \begin{align} \label{eq:T formula j=0}
            \calH[\calT_0^{n_1,\cdots,n_l}[f]] = \sum_{k=1}^l \frac{n_k}{n} \calH[\calT_k^{n_1,\cdots,n_k-1,\cdots,n_l}[f]],
        \end{align}
        and for $j \in \{1,\dots,l\}$ and $n_1,\cdots,n_l \geq 1$, 
        \begin{align} 
            \begin{split} \label{eq:T formual j!=0}
                \calH[\calT_j^{n_1,\cdots,n_l}[f]] 
                & = \sum_{k=1}^l \frac{n_k}{n} \calH[f'_j \cdot \calT_k^{n_1,\cdots,n_k-1,\cdots,n_l}[f]] \\
                & \quad -2f'_j  \sum_{k=1}^l \frac{n_k}{n} \calH[\calT_k^{n_1,\cdots,n_k-1,\cdots,n_l}[f]] + \calH[f'_j]\calT_0^{n_1,\cdots,n_l}[f]\\
                & \quad + J_{\mathrm{rem}},
            \end{split}
        \end{align}
        where $n = \sum_{k=1}^l n_k$ and the remainder term $J_{\mathrm{rem}}$ is given by
        \begin{align*}
            J_{\mathrm{rem}}(x) & := \frac{1}{\pi} P.V.\int_\bbR \frac{f'_j(x) - f'_j (z)}{x-z} (\calN(x,x) - \calN(x,z)) \ud z, \\
            \calN(x,z) & := \calH_1[\Pi\calQ [f]^{n_1,\cdots,n_l}](x,z).
        \end{align*}
    \end{lemma}
    We now apply this lemma to \eqref{eq:prop2_obj operator} with $l=2$ and $f=(\zt_1,\zt_2)$. Using the notation of \eqref{eq:def T}, we can express $\calH[S_{1,n}]$ as 
    \begin{align*}
        \calH[S_{1,n}] = \sum_{n_1,n_2} C_{n_1,n_2}^n \calH[\calT_0^{n_1,n_2}[f]].
    \end{align*}
    Each summand on the right-hand side belongs to $C^\vphi$. Hence, to conclude that $\calH[S_{1,n}] \in C^\vphi$, it is sufficient to show that the norms of each summand are summable. We now introduce a lemma, which provides decay estimates for the norm of $\calH[\calT_0^{n_1,n_2}[f]]$.

    \begin{lemma} [Decay estimates of {$\calH[\calT_j^{n_1,n_2}[f]]$}]\label{Lem:T2}
        Suppose that $f \in C^{1,\vphi}(\bbR;\bbR^2)$ is supported in $[-\dlt,\dlt]$, and that its Hilbert transform $\calH[f]$ also belongs to $C^{1,\vphi}$. Assume further that the following bounds hold:
        \begin{align*}
            \begin{split}
                \nrm{f}_{C^1} \leq M \vphi(\dlt),\quad
                \nrm{f}_{C^{1,\vphi}} \aleq M,\quad
                \nrm{\calH[f]}_{C^1} \aleq M \vphi(\dlt)^{1-\tht},\quad
                \nrm{\calH[f]}_{C^{1,\vphi}} \aleq M \vphi(\dlt)^{-\tht}.
            \end{split}
        \end{align*}
        Then, for sufficiently small $\dlt>0$, we have exponential bounds on the norm of $\calH[\calT_j^{n_1,n_2}[f]]$. That is, for some constant $C(\vphi)$, for $\dlt < C(\vphi)$, 
        \begin{align}
            \begin{split}\label{eq:HT2 nrm}
                \nrm{\calH[\calT_j^{n_1,n_2}[f]]}_{C^\vphi} \aleq \vphi(\dlt)^{C(\tht)} (M\alp_1)^{n_1+n_2},
            \end{split}
        \end{align} 
        where $\alp_1=C_{P,1}\vphi(\dlt)$ and $C(\tht)$ denotes some constants depending only on $\tht$.
    \end{lemma}

    By applying the lemma to $\calH[S_{1,n}]$, we have 
    \begin{align*}
        \begin{split}
            \nrm{\calH[S_{1,n}]}_{C^\vphi}
            & \leq \sum_{n_1,n_2} C_{n_1,n_2}^n \nrm{\calH[\calT_0^{n_1,n_2}[f]]}_{C^\vphi} \\
            & \aleq \sum_{n_1,n_2} C_{n_1,n_2}^n \vphi(\dlt)^{C(\tht)} (M\alpha_1)^{n_1+n_2} \\
            & \aleq \vphi(\dlt)^{C(\tht)} \left( (c_1+c_2) (M\alpha_1) + (c_{11}+c_{12}+c_{22}) (M\alpha_1) \right)^n,
        \end{split}
    \end{align*}
    where the last inequality follows from the multinomial expansion. Hence by taking $\dlt$ small enough and thus $\alpha_1 = C_{P,1} \vphi(\dlt)^{1-2\tht}$ small enough, we obtain the summability of the above terms.

    \noindent\textbf{$\bullet$ Estimates on $\calH[S_4]$.} \\
    We recall that $S_{4,n}$ is of the form
    \begin{align*}
        \begin{split}
            S_{4,n} 
            & = (-1)^n P.V.\int_\bbR \frac{\rd_\xi \zt(\eta)}{\xi-\eta}  \calQ[\zt_1](\xi,\eta) \calD^n(\xi,\eta) \ud \eta.
        \end{split}
    \end{align*}
    Adopting the notation in \ref{Lem:T}, we have that 
    \begin{align*}
        \begin{split}
            S_{4,n} 
            & = \sum_{n_1,n_2} C_{n_1,n_2}^n P.V.\int_\bbR \frac{\rd_\xi \zt(\eta)}{\xi-\eta} \left( \Pi\calQ[\zt]^{n_1+1,n_2} \right)(\xi,\eta) \ud \eta \\
            & = \sum_{n_1,n_2} C_{n_1,n_2}^n (\calT_1^{n_1+1,n_2}[f],\calT_2^{n_1+1,n_2}[f])^t (\xi),
        \end{split}
    \end{align*}
    where the parentheses denote the vector. Using the same method as in the previous step, by applying Lemma \ref{Lem:T2}, we obtain the estimate of the norm of $\calH[S_{4,n}]$:
    \begin{align*}
        \begin{split}
            \nrm{\calH[S_{4,n}]}_{C^\vphi}
            & \aleq \sum_{n_1,n_2} C_{n_1,n_2}^n \nrm{(\calH[\calT_1^{n_1+1,n_2}[f]],\calH[\calT_2^{n_1+1,n_2}[f]])^t}_{C^\vphi} \\
            & \aleq \sum_{n_1,n_2} C_{n_1,n_2}^n \vphi(\dlt)^{C(\tht)} (M\alp_1)^{n_1+n_2+2} \\
            & \aleq M^2\vphi(\dlt)^{C(\tht)} ((c_1+c_2)(M\alp_1)+(c_{11}+c_{12}+c_{22})(M\alp_1))^n.
        \end{split}
    \end{align*}
    By taking small enough $\dlt$ such that the above series converges, we conclude that $\calH[S_4] \in C^\vphi$. A similar argument shows that $\calH[S_2],\calH[S_3],\calH[S_5] \in C^\vphi$.
\end{proof}

\subsection{Lipschitz Estimates for the Velocity Field and Its Hilbert Transform} \label{Subsec:Lip Estimates of v, Hv}
In this subsection, we establish the Lipschitz estimates for the velocity field and its Hilbert transform, which constitute the last ingredient of the fixed-point argument. Throughout this subsection, superscripts in parentheses indicate the index of functions, and $\Dlt$ denotes the difference of functions
\begin{align*}
    \Dlt[f] = f^{(1)} - f^{(2)}
\end{align*}

\begin{proposition} \label{Prop:Lip.est of v}
    Let $(\gma^{(a)},\calH[\gma^{(a)}]) \in \bfO^M, a=1,2$. We denote their distance in $\bfB$ by $\rho:=\nrm{\Dlt[ (\gma,\calH[\gma]) ]}_{C^{1,\vphi}}$. Then 
    \begin{align} \label{eq:prop3_obj1}
        \nrm{v[\gma^{(1)}]-v[\gma^{(2)}]}_{C^{1,\vphi}} \aleq C(M)\rho.
    \end{align}
\end{proposition}
\begin{proof}[Proof of Proposition \ref{Prop:Lip.est of v}]
    As before, we prove \eqref{eq:prop3_obj1} by pointwise estimates of the norm. In this proof, we adopt the notations and representations used in the previous subsection. Also, functions related to $\gma$ with parenthesized superscripts denote the quantities associated with $\gma^{(\cdot)}$. Let us fix $\tht = 0$. Since $(\gma_i,\calH[\gma_i]) \in B^M$, we can take $\delta$ uniformly and use the same decomposition for each $\rd_\xi v^{(a)}$.
    We begin by presenting the bounds on the difference $\Dlt[\zt]$:
    \begin{align*}
        \begin{split}
            \nrm{\Dlt[\zt]}_{C^0} \aleq \rho\dlt\vphi(\dlt),\quad
            \nrm{\Dlt[\zt]}_{C^1} \aleq \rho\vphi(\dlt),\quad
            \nrm{\Dlt[\zt]}_{C^{1,\vphi}} \aleq \rho\\
            \nrm{\calH[\Dlt[\zt]]}_{C^1} \aleq \rho\vphi(\dlt)^{1-\tht},\quad
            \nrm{\calH[\Dlt[\zt]]}_{C^{1,\vphi}} \aleq \rho\vphi(\dlt)^{-\tht}.
        \end{split}
    \end{align*}
    
    \noindent\textbf{STEP 1: Estimates for the Regular Part}

    First, we recall the representation of $\Dlt[\rd_\xi v_r[\gma]]$:
    \begin{align*}
        \Dlt[\rd_\xi v_r[\gma]] = \int_{|\xi-\eta|>\frac{\delta}{4}} \Dlt \left[ \frac{(\gma^{(\cdot)} (\xi)-\gma^{(\cdot)} (\eta))\cdot\gma^{(\cdot)} (\xi)}{|\gma^{(\cdot)} (\xi)-\gma^{(\cdot)} (\eta)|^2 } \rd_\xi \gma^{(\cdot)} (\eta) \right] (1-\chi(\xi-\eta)) \ud \eta.
    \end{align*}
    Using a telescoping sum, we obtain the bounds on the integrand:
    \begin{align*}
        \left\Vert \frac{\ud}{\ud\xi} \Dlt \left[ \frac{(\gma^{(\cdot)} (\xi)-\gma^{(\cdot)} (\eta))\cdot\gma^{(\cdot)} (\xi)}{|\gma^{(\cdot)} (\xi)-\gma^{(\cdot)} (\eta)|^2 } \rd_\xi \gma^{(\cdot)}(\eta) \right] \right\Vert_{C^0} \aleq C(M) \rho.
    \end{align*}
    which implies that
    \begin{align*}
        \nrm{\Dlt[\rd_\xi v_r[\gma]]}_{C^\alpha} \aleq C(M,\alpha) \rho.
    \end{align*}

    \noindent\textbf{STEP 2: Estimates for the Singular Part}

    In this step, as we have done before, we may assume that $|\xi|<\dlt/4$ such that 
    \begin{align*}
        \rd_\xi v_s[\gma] = \calV[\zt].
    \end{align*}
    The estimates on $\Dlt[\calV_2[\zt]]$ can be readily obtained. Hence we focus on the difference of $\calV_1[\zt]$. Applying the telescoping sum on \eqref{eq:decomposition of v_s} yields 
    \begin{align*}
        \Dlt[ \calV_1[\zt] ]
        & = \sum_{i=1}^5 \Dlt[P_i(U,\tld{\xi},\zt(\xi))] \cdot S_i^{(2)} + \sum_{i=1}^5 P_i(U^{(1)},\tld{\xi},\zt^{(1)}(\xi)) \Dlt[S_i].
    \end{align*} 
    Since $\Dlt[U] \leq C(M) \rho$, we have the following bound for the first term:
    \begin{align*}
        \left\Vert\sum_{i=1}^5 \Dlt[P_i(U,\tld{\xi},\zt(\xi))] \cdot S_i^{(2)}\right\Vert_{C^\vphi} 
        & \aleq \sum_{i=1}^5 \nrm{\Dlt[P_i(U,\tld{\xi},\zt(\xi))]}_{C^\vphi} \cdot \nrm{S_i^{(2)}}_{C^\vphi} \\      
        & \aleq C(M) \rho.
    \end{align*} 
    Hence it suffices to obtain the same bound for the second term and this will follow from the estimate
    \begin{align*}
        \nrm{\Dlt[S_i]}_{C^\vphi} \leq C(M) \rho,
    \end{align*}
    which we present in the next step. 

    \noindent\textbf{STEP 3: Estimates for $\Dlt[\calD^n]$}

    To obtain a bound on $\Dlt[S_i]$, we begin by establishing bounds on $\Dlt[\calD^n]$. Applying a telescoping sum to $\Dlt[\calD^n]$ yields the following decomposition: 
    \begin{align*}
        \Dlt[ \calD^n ] = \Dlt[\calD] \times \sum_{k=0}^{n-1} (\calD^{(1)})^k (\calD^{(2)})^{n-k-1}.
    \end{align*}
    From the estimates on $\calD$ \eqref{eq:D^n nrm}, we obtain
    \begin{align} \label{eq:prop3_nrm Dlt D}
        \begin{split}
            \nrm{\Dlt[\calD^n]}_{C^\vphi}
            & \leq \nrm{\Dlt[\calD]}_{C^\vphi} \times \sum_{k=0}^{n-1} \nrm{(\calD^{(1)})^k (\calD^{(2)})^{n-k-1}}_{C^\vphi} \\
            & \aleq n \rho \vphi(\dlt)^{-1} (M\vphi(\dlt))^{n-1}.
        \end{split}
    \end{align}
    Similarly, we can estimate the norm of its Hilbert transform using \eqref{eq:Hfg nrm_simple} and \eqref{eq:Hf^n nrm}
    \begin{align} \label{eq:prop3_nrm calH Dlt D}
        \begin{split}
            \lVert\calH [ \Dlt[\calD^n] ]\rVert_{C^\vphi}
            & \leq \nrm{\calH[\Dlt[\calD]]}_{C^\vphi} \left( \sum_{k=0}^{n-1} \nrm{\calH[(D^{(1)})^k]}_{C^\vphi} \nrm{\calH[(\calD^{(2)})^{n-k-1}]}_{C^\vphi} \right) \\
            & \aleq \rho \vphi(\dlt)^{-\tht} \cdot \sum_{k=0}^{n-1} \nrm{\calH[(\calD^{(1)})^k]}_{C^\vphi} \nrm{\calH[(\calD^{(2)})^{n-k-1}]}_{C^\vphi} \\
            & \aleq n\rho\vphi(\dlt)^{C(\tht)} (M\alpha_1)^{n-1}.
        \end{split} 
    \end{align}

    \noindent\textbf{$\bullet$ Estimates for $\nrm{\Dlt[S_1]}_{C^\vphi}$} 

    Recall that $\Dlt[S_1]$ is defined as follows: 
    \begin{align*}
        \Dlt[S_1] = \sum_{n=0}^\infty (-1)^n P.V.\int_\bbR \frac{1}{\xi-\eta} \Dlt[\calD^n](\xi,\eta) \ud \eta.
    \end{align*}
    By applying \eqref{eq:prop3_nrm calH Dlt D}, we have the following norm estimates: 
    \begin{align*}
        \begin{split}
            \nrm{\Dlt[S_1]}_{C^\vphi}
            & \leq \sum_{n=0}^\infty \nrm{\calH_2[\Dlt[\calD^n]]}_{C^\vphi} \\
            & \aleq \rho \vphi(\dlt)^{C(\tht)} \sum_{n=1}^\infty n (M\alpha_1)^{n-1} \\
            & \aleq C(M) \rho.
        \end{split}
    \end{align*}
    In the last inequality, we assume that $\dlt$ is small enough so that the sum converges. 

    \noindent\textbf{$\bullet$ Estimates for $\nrm{\Dlt[S_2]}_{C^\vphi}$}

    Applying a telescoping sum to $\Dlt[S_2]$, we have the following decomposition: 
    \begin{align*}
        \begin{split}
            \Dlt[S_2] 
            & = \sum_{n=0}^\infty (-1)^n P.V.\int_\bbR \frac{\Dlt[\rd_\xi \zt](\eta)}{\xi-\eta} (\calD^{(2)})^n(\xi,\eta) \ud \eta \\
            & \quad + \sum_{n=0}^\infty (-1)^n P.V.\int_\bbR \frac{\rd_\xi \zt^{(1)}(\eta)}{\xi-\eta} \Dlt[\calD^n](\xi,\eta) \ud \eta \\
            & = : \Dlt S_{2,1} + \Dlt S_{2,2}.  
        \end{split}      
    \end{align*}
    For $\Dlt S_{2,1}$, we apply \eqref{eq:Afg nrm} and \eqref{eq:D^n nrm} to obtain
    \begin{align*}
        \begin{split}
            \nrm{\Dlt S_{2,1}}_{C^\vphi} 
            & \leq \sum_{n=0}^\infty (\nrm{\Dlt[\rd_\xi \zt]}_{C^\vphi} + \nrm{\calH[\Dlt[\rd_\xi \zt]]}_{C^\vphi}) (\nrm{(\calD^{(2)})^n}_{C^\vphi} + \nrm{\calH[(\calD^{(2)})^n]}_{C^\vphi} ) \\
            & \leq \rho \vphi(\dlt)^{C(\tht)} \sum_{n=0}^\infty (M\alp_1)^n \\
            & \leq C(M) \rho.
        \end{split}
    \end{align*}
    For $\Dlt S_{2,2}$, we use \eqref{eq:prop3_nrm Dlt D} and \eqref{eq:prop3_nrm calH Dlt D} to obtain 
    \begin{align*}
        \begin{split}
            \nrm{\Dlt S_{2,2}}_{C^\vphi} 
            & \leq \sum_{n=0}^\infty (\nrm{\rd_\xi \zt^{(1)}}_{C^\vphi} + \nrm{\calH[\rd_\xi \zt^{(1)}]}_{C^\vphi}) (\nrm{\Dlt[(\calD^{(\cdot)})^n]}_{C^\vphi} + \nrm{\calH[\Dlt[\calD^n]]}_{C^\vphi} ) \\
            & \leq M\vphi(\dlt)^{-\tht} \sum_{n=0}^\infty (\nrm{\Dlt[(\calD)^n]}_{C^\vphi} + \nrm{\calH[\Dlt[\calD^n]]}_{C^\vphi} ) \\
            & \leq M\vphi(\dlt)^{-\tht} \rho \sum_{n=0}^\infty n(M\alp_1)^{n-1} \\
            & \leq C(M) \rho.
        \end{split}
    \end{align*}

    \noindent\textbf{$\bullet$ Estimates for $\nrm{\Dlt[S_4]}_{C^\vphi}$}

    Recall that 
    \begin{align*}
        S_4 = \sum_{n=0}^\infty P.V.\int_\bbR \frac{1}{\xi-\eta} \calD^n(\xi,\eta) \calQ[\zt_1](\xi,\eta) \rd_\xi \zt(\eta) \ud \eta.
    \end{align*}
    Using a telescoping sum, we obtain the following decomposition:
    \begin{align*}
        \begin{split}
            \Dlt[S_4]
            & = \sum_{n=0}^\infty P.V.\int_\bbR \frac{1}{\xi-\eta} \Dlt[\calD^n](\xi,\eta) \calQ[\zt_1^{(2)}](\xi,\eta) \rd_\xi \zt^{(2)} (\eta) \ud \eta \\
            & \quad + \sum_{n=0}^\infty P.V.\int_\bbR \frac{1}{\xi-\eta} (\calD^{(1)} )^n(\xi,\eta) \Dlt[\calQ[\zt_1 ]](\xi,\eta) \rd_\xi \zt^{(2)} (\eta) \ud \eta \\
            & \quad + \sum_{n=0}^\infty P.V.\int_\bbR \frac{1}{\xi-\eta} (\calD^{(1)} )^n(\xi,\eta) \calQ[\zt_1^{(1)} ](\xi,\eta) \Dlt[\rd_\xi \zt ](\eta) \ud \eta.
        \end{split}
    \end{align*}
\end{proof}

\begin{proposition} \label{Prop:lip.est of Hv}
    Let $(\gma^{(a)},\calH[\gma^{(a)}]) \in \bfO^M, a=1,2$. Denote the distance between them in $\{C^{1,\vphi}\}^2$ by $\rho = \nrm{\Dlt[(\gma,\calH[\gma])]}_{C^{1,\vphi}}$. Then 
    \begin{align}
        \nrm{\calH[v[\gma^{(a)}]]-\calH[v[\gma^{(b)}]]}_{C^{1,\vphi}} \aleq C(M)\rho.
    \end{align}
\end{proposition}
\begin{proof}[{Proof of Proposition \ref{Prop:lip.est of Hv}}]

    As we have established, the difference of the regular part of the velocity field belongs to $C^{1,\alpha}$ and satisfying
    \begin{align*}
        \nrm{\Dlt[\rd_\xi v_r[\gma]]}_{C^\alpha} \leq C(M,\alpha) \rho.
    \end{align*}
    Hence, we get 
    \begin{align*}
        \nrm{\calH[\Dlt[\rd_\xi v_r[\gma]]]}_{C^\alpha} \leq C(M,\alpha) \rho.
    \end{align*}

    \noindent\textbf{STEP 1: Estimates for Singular Part}

    In this step, we establish the $C^{\vphi}$ estimates for $\calH[\rd_\xi v_s[\gma]]$. First, assume that $|\xi|<\dlt/8$. Since the other term can be bounded similarly, we only focus on $\calH[\calV_1[\zt]]$. Recall that 
    \begin{align*}
        \begin{split}
            \Dlt[\calV[\zt]]
            & = \sum_{i=1}^4 P_i(\Dlt[U],\tld{\xi},\zt) S_i^{(2)} \\
            & \quad + \sum_{i=1}^4 P_i(U^{(1)},\tld{\xi},\zt) \Dlt[S_i].
        \end{split}
    \end{align*} 
    We proved that $\nrm{\Dlt P_i(U,\tld{\xi})}_{C^\alpha} \leq C(M',\alpha) \rho$ and $\nrm{S_i}_{C^\vphi} \leq C(M)$. Hence, we find that the desired Lipschitz estimate holds in the first term. Also, by the product lemma, to obtain the desired bound for the second term, it suffices to show that
    \begin{align*}
        \nrm{\calH[\Dlt[ S_i ]]}_{C^\vphi} \leq C(M) \rho.
    \end{align*}

    \noindent\textbf{STEP 3: Estimates for $\Dlt[\calH[S_1]]$}

    In this step, we establish the desired bound on $\Dlt[\calH[S_1]]$. As in the proof of Proposition 2, it is crucial to exploit the algebraic structure of $S_1$. First, we apply a telescoping sum to $\calH[\Dlt[S_1]]$ to obtain
    \begin{align*}
        \begin{split}
            \calH[\Dlt[ S_1 ] ] 
            & = \sum_{n=0}^\infty (-1)^n \calH \left[ \int_\bbR \frac{1}{\xi-\eta} \Dlt[\calD^n](\xi,\eta) \ud \eta \right] \\
            & = \sum_{n=0}^\infty (-1)^n \calH \left[ \int_\bbR \frac{1}{\xi-\eta} \Dlt[\calD](\xi,\eta) \left( \sum_{k=0}^n (\calD^{(1)})^k (\calD^{(2)})^{n-1-k} \right)(\xi,\eta) \ud \eta \right] \\
            & = \sum_{n=0}^\infty (-1)^n \sum_{k=0}^n  \calH \left[ \int_\bbR \frac{1}{\xi-\eta} \Dlt[\calD](\xi,\eta) \left( (\calD^{(1)})^k (\calD^{(2)})^{n-1-k} \right)(\xi,\eta) \ud \eta \right].
        \end{split}
    \end{align*}
    Recall the expression of $\calD$ in the proof of proposition 2.
    \begin{align*}
        \calD = c_1\calQ[\zt_1] + c_2\calQ[\zt_2] + c_{11}\calQ[\zt_1]^2 + c_{12}\calQ[\zt_1] \calQ[\zt_2] + c_{22} \calQ[\zt_2]^2.
    \end{align*}
    Before directly taking the difference of $\calD$, for notational convenience, we define the vector-valued function $f : \bbR \to \bbR^6$ as
    \begin{align*}
        f := (\Dlt[\zt_1],\Dlt[\zt_2],\zt_1^{(1)},\zt_2^{(1)},\zt_1^{(2)},\zt_2^{(2)}).
    \end{align*}
    Now, taking the difference of $\calD$ and adopting the definition of $f$ above, we have
    \begin{align*}
        \begin{split}
            \Dlt[\calD] & = c_1 \calQ[f_1] + c_2 \calQ[f_2]  + c_{11} \calQ[f_1](\calQ[f_3]+\calQ[f_4]) \\
            & \quad + c_{12} (\calQ[f_1]\calQ[f_4] + \calQ[f_2]\calQ[f_5]) + c_{22} \calQ[f_2](\calQ[f_5]+\calQ[f_6]).
        \end{split} 
    \end{align*}
    Then, we conclude that $\calH[\Dlt[S_{1,n}]]$ is the sum of the following operators
    \begin{align*}
        & \sum_{k=0}^n \calH\left[ \int_\bbR \frac{1}{\xi-\eta} \calQ[f_j] g \left( (\calD^{(1)})^k (\calD^{(2)})^{n-k} \right) \ud \eta\right] \\
        & = \sum_{k=0}^n \calH\left[ \int_\bbR \frac{1}{\xi-\eta} \calQ[f_j] g \left( \left( \sum_{n_1,n_2} C_{n_1,n_2}^k \calQ[f_3]^{n_1}\calQ[f_4]^{n_2} \right)\times \left( \sum_{n_3,n_4} C_{n_3,n_4}^{n-k} \calQ[f_5]^{n_3}\calQ[f_6]^{n_4} \right) \right) \ud \eta \right],
    \end{align*}
    for $j \in \{1,2\}$ and $g= 1, \calQ[f_3],\cdots,\calQ[f_6]$. Consequently, dealing with $\calH[\Dlt[S_1^{(\cdot)}]]$ is reduced to considering the operators
    \begin{align*}
        \calH \left[ \int_\bbR \frac{1}{\xi-\eta} \calQ[f_j] \Pi\calQ[f]^{0,0,n_3,\cdots,n_6} \ud \eta \right],
    \end{align*}
    which is indeed $\calH[\calT_0^{1,0,n_3,\cdots,n_6}[f]]$ or $\calH[\calT_0^{0,1,n_3,\cdots,n_6}[f]]$. We now introduce a lemma that provides bounds for such operators.

    \begin{lemma} [Decay estimates of $\cdots$] \label{Lem:T6}
        Suppose that a function $f \in C^\vphi(\bbR;\bbR^6)$ is supported in $[-2\dlt,2\dlt]$, and that its Hilbert transform $\calH[f]$ also belongs to $C^\vphi$. We further assume that the following bounds hold:\\
        for $k=1,2$, 
        \begin{align*}
            \begin{split}
                \nrm{f_k}_{C^0} \aleq \rho \dlt \vphi(\dlt),\quad \nrm{f_k}_{C^1} \aleq \rho \vphi(\dlt),\quad \nrm{f_k}_{C^\vphi} \aleq \rho \\
                \nrm{\calH[f_k]}_{C^1} \aleq \rho \vphi(\dlt)^{1-\tht},\quad \nrm{\calH[f_k]}_{C^\vphi} \aleq \rho \vphi(\dlt)^{-\tht},
            \end{split}
        \end{align*}
        and for $k=3,4,5,6$, 
        \begin{align*}
            \begin{split}
                \nrm{f_k}_{C^0} \aleq M\dlt \vphi(\dlt),\quad \nrm{f_k}_{C^1} \aleq M \vphi(\dlt),\quad \nrm{f_k}_{C^{1,\vphi}} \aleq M\\
                \nrm{\calH[f_k]}_{C^1} \aleq M \vphi(\dlt)^{1-\tht},\quad \nrm{\calH[f_k]}_{C^{1,\vphi}} \aleq M \vphi(\dlt)^{-\tht}.
            \end{split}
        \end{align*}
        Then for small enough $\dlt < 1$, we obtain exponential decay on the norms of $\calT$ and its Hilbert transform:
        For $j\in\{1,2\}$,
        \begin{align}
            \begin{split}
                \nrm{\calH[\calT_j^{0,0,n_3,\cdots,n_6}[f]]}_{C^\vphi} \aleq \rho \vphi(\dlt)^{C(\tht)} (M\alp_1)^{n},
            \end{split}
        \end{align}
        and for $j\in\{3,4,5,6\}$ and $(n_1,n_2) \in \{(1,0),(0,1)\}$, 
        \begin{align} \label{eq:HT6 nrm}
            \nrm{\calH[\calT_j^{n_1,n_2,n_3,\cdots,n_6}[f]]}_{C^\vphi}\aleq \rho \vphi(\dlt)^{C(\tht)} (M\alp_1)^{n+1},
        \end{align}
        where $\alp_1 = C_{P,1} \vphi(\dlt)^{1-2\tht}$ and $n = n_3 + \cdots n_6$ and $C(\tht)$ denote some constants depending only on $\tht$, which may differ from line to line.
    \end{lemma}
    
    Note that the assumptions of the above lemma are tailored to our $f$. We now apply \eqref{eq:HT6 nrm} to $\calH[\Dlt[S_{1,n}^{(\cdot)}]]$
    \begin{align*}
        \nrm{\calH[\Dlt[S_{1,n}^{(\cdot)}]]}_{C^\vphi}
        & \aleq \sum_{k=0}^n \sum_{n_1,n_2} C_{n_1,n_2}^k \sum_{n_3,n_4} C_{n_3,n_4}^{n-k} \nrm{\calH[\calT_0^{1,0,n_3+1,n_4,n_5,n_6}[f]]}_{C^\vphi} \\
        & \aleq \sum_{k=0}^n \sum_{n_1,n_2} C_{n_1,n_2}^k \sum_{n_3,n_4} C_{n_3,n_4}^{n-k} \rho \vphi(\dlt)^{C(\tht)} (M\alp_1)^{n_3+\cdots+n_6 + 1} \\
        & \aleq \rho \vphi(\dlt)^{C(\tht)} ((c_1+c_2)(M\alp_1)+ (c_{11}+c_{12}+c_{22})(M\alp_2)^2)^{n+1}.
    \end{align*}
    Applying Lemma \ref{Lem:T6} and a method similar to that in the proof of Proposition 2, we obtain, we obtain that $\nrm{(3.76)}_{C^\vphi} \leq C(M',\vphi) \rho \alpha^n $ for some $\beta_1 < 1$. Similarly, we obtain that 
    \begin{align*}
        & \nrm{\calH[\Dlt[S_{1,n}^{(\cdot)}]]} \leq C(M') \rho \beta^n \\
        & \nrm{\calH[\Dlt[S_{1}^{(\cdot)}]]} \leq C(M') \rho .
    \end{align*}
    One can estimate the difference of $\calH[S_i]$ in the similar way to apply lemma in appropriate way.

\end{proof}

\subsection{Proof of Theorem \ref{Thm2}}
We now prove Theorem \ref{Thm2}.
\begin{proof} [Proof of Theorem \ref{Thm2}]
    Given initial data $\Omg_0$, let $\gma_0 : \bbT \to \bbR^2$ be a $C^{1,\psi_0}$ parametrization of $\rd\Omg_0$ satisfying the non-degeneracy
    \begin{align*}
        |\gma_0|_\star := \inf_{\xi\neq\eta} \frac{|\gma(\xi)-\gma(\eta)|}{|\xi-\eta|} > 0.
    \end{align*} 
    Then by Lemma \ref{Lem:img of H}, $\calH[\gma_0]$ belongs to $C^{1,\psi}(\bbT)=C^{1,\tld{\psi_0}}(\bbT)$. We now apply Theorem \ref{Thm1} to the pair $(\gma_0,\calH[\gma_0])$. Then for $T = T(\psi, \nrm{\gma_0}_{C^{1,\psi}}, \nrm{\calH[\gma_0]}_{C^{1,\psi}}) = T(\psi,\nrm{\gma_0}_{C^{1,\psi}})>0$, we obtain the solution $(\gma,\calH[\gma]): \bbT \times [0,T] \to \bbR^2\times\bbR^2$ to \eqref{eq:HCDE}.
    
    Since $\gma(t)\in C^{1,\psi}$ is the solution to \eqref{eq:CDE} for $t \in [0,T]$ and $|\gma(\cdot,t)|_\star > 0$ for $t \in [0,T]$, we conclude that the vortex patch remains a $C^{1,\tld{\psi}}$ domain for time $t \in [0,T]$. 
\end{proof}

\section{Proofs of Lemmas} \label{Sec:Proof of Lem} \addtocontents{toc}{\protect\setcounter{tocdepth}{1}}
In this section, we prove the lemmas appearing in the previous sections. We also assume that $\vphi$ is a modulus of continuity satisfying \ref{A1}--\ref{A3}.

\subsection{Proof of Lemma \ref{Lem:img of H}} \label{Subsec:img of H}
\begin{proof}[Proof of Lemma \ref{Lem:img of H}]
    First, we obtain the $C^0$ bounds for $\calH[f]$:  
    \begin{align*}
        \begin{split}
            |\calH[f](x)|
            & \leq \frac{1}{\pi} \int_{\supp(f)} |f(x)-f(y)|^{1-\tht} \frac{|f(x)-f(y)|^\tht}{|x-y|} \ud y \\
            & \leq \frac{1}{\pi} \nrm{f}_{C^0}^{1-\tht} \nrm{f}_{C^\vphi}^\tht \int_{A(x)} \frac{\vphi^\tht(r)}{r} \ud r \\
            & \aleq \nrm{f}_{C^0}^{1-\tht} \nrm{f}_{C^\vphi}^\tht,
        \end{split}
    \end{align*}
    where $A(x) := \{ r \in \bbR^+ : \dist(x,\supp(f)) < r < \dist(x,\supp(f)) + 2 \diam(\supp(f)) \}$ and the last inequality follows from \ref{A1}. Next, we establish the $\dot{C}^{\tld{\vphi}}$ estimate. For $h>0$ sufficiently small, consider the following decomposition:
    \begin{align*}
        \begin{split}
            \calH[f](x+h) - \calH[f](x)
            & = \frac{1}{\pi} \int_{|x-y|<2h} \frac{f(x)-f(y)}{x-y} \ud y \\
            & \quad - \frac{1}{\pi} \int_{|x-y|<2h} \frac{f(x+h)-f(y)}{x+h-y} \ud y \\
            & \quad + \frac{1}{\pi} \int_{|x-y|>2h} \left( \frac{1}{x-y} - \frac{1}{x+h-y} \right) (f(x+h)-f(y)) \ud y \\
            & \quad + \frac{1}{\pi} \int_{|x-y|>2h} \frac{f(x+h) - f(x)}{x-y} \ud y \\
            & =: I_1 + I_2 + I_3 + I_4.
        \end{split}
    \end{align*}
    By symmetry, the last term $I_4$ vanishes. For $I_1$, we estimate 
    \begin{align*}
        \begin{split}
            |I_1| 
            & \leq \frac{1}{\pi} \int_{|x-y|<2h} \left|\frac{f(x)-f(y)}{x-y}\right| \ud y \\
            & \leq \frac{1}{\pi} \nrm{f}_{C^\vphi} \int_{|x-y|<2h} \frac{\vphi(|x-y|)}{|x-y|} \ud y \\
            & \aleq \nrm{f}_{C^\vphi} \tld{\vphi}(h).
        \end{split}
    \end{align*}
    An identical estimate holds for $I_2$:
    \begin{align*}
        |I_2| \aleq \nrm{f}_{C^\vphi} \tld{\vphi}(h).
    \end{align*}
    For $I_3$, we apply the mean value theorem on the kernel to obtain
    \begin{align} \label{eq:bound on kernel}
        \left| \frac{1}{x-y} - \frac{1}{x+h-y} \right| \leq 10 \frac{h}{|x-y|^2} \quad \text{ for $|x-y|>2h$}.
    \end{align}
    Then \eqref{eq:bound on kernel} and $I_3$ yield
    \begin{align*}
        \begin{split}
            |I_3|
            & \aleq \int_{|x-y|>2h} \left| \frac{1}{x-y} - \frac{1}{x+h-y} \right| |f(x+h)-f(y)| \ud y \\
            & \aleq \nrm{f}_{C^\vphi} h \int_{2h<|x-y|<2} \frac{1}{|x-y|} \frac{\vphi(|x-y|)}{|x-y|} \ud y \\
            & \aleq \nrm{f}_{C^\vphi} h \frac{\vphi(2h)}{2h} \int_{2h<|x-y|<C} \frac{1}{|x-y|} \ud y \\
            & \aleq \nrm{f}_{C^\vphi} \vphi(h) \left(-\log\frac{C}{h}\right) \\
            & \aleq \nrm{f}_{C^\vphi} \tld{\vphi}(h),
        \end{split}
    \end{align*}
    where in the last inequality we used \ref{A3}. Combining these estimates, we conclude the desired result:
    \begin{align*}
        \nrm{\calH[f]}_{C^{\tld{\vphi}}} \aleq \nrm{f}_{C^\vphi}.
    \end{align*}
\end{proof}

\subsection{Proof of Lemma \ref{Lem:Hfg}} \label{Subsec:Hfg}
\begin{proof}[Proof of Lemma \ref{Lem:Hfg}]
    First, the $C^0$ norm of $\calH[fg]$ is bounded by \eqref{eq:Hf C^0}
    \begin{align*}
        \nrm{\calH[fg]}_{C^0} \aleq \nrm{fg}_{C^0}^{1-\tht} \nrm{fg}_{C^\vphi}^{\tht}.
    \end{align*}
    By Cotlar identity \eqref{eq:Cotlar}, we obtain the following representation of $\calH[fg]$: 
    \begin{align*}
        \begin{split}
            \calH[fg](x) & = \calH[f](x) \cdot g(x) + f(x) \cdot \calH[g](x) + \calH[\calH[f]\calH[g]](x) \\
            & = -\frac{1}{\pi} P.V.\int_\bbR \frac{(\calH[f](x) - \calH[f](y)) (\calH[g](x) - \calH[g](y))}{x-y} \ud y.
        \end{split}
    \end{align*}
    We now establish the $C^\vphi$ estimate of $\calH[fg]$ using the above formula. For $h>0$ sufficiently small, we decompose the difference of $\calH[fg]$ as
    \begin{align*}
        \begin{split}
            & \calH[fg](x+h) - \calH[fg](x)\\
            & = \frac{1}{\pi} P.V.\int_\bbR \frac{(\calH[f](x) - \calH[f](y)) (\calH[g](x) - \calH[g](y))}{x-y} \ud y \\
            & \quad - \frac{1}{\pi} P.V.\int_\bbR \frac{(\calH[f](x+h) - \calH[f](y)) (\calH[g](x+h) - \calH[g](y))}{x+h-y} \ud y\\
            & = \frac{1}{\pi} \int_{|x-y|<2h} \frac{(\calH[f](x) - \calH[f](y)) (\calH[g](x) - \calH[g](y))}{x-y} \ud y\\
            & \quad - \frac{1}{\pi} \int_{|x-y|<2h} \frac{(\calH[f](x+h) - \calH[f](y)) (\calH[g](x+h) - \calH[g](y))}{x+h-y} \ud y \\
            & \quad + \frac{1}{\pi} P.V.\int_{|x-y|>2h} \left(\frac{1}{x-y}-\frac{1}{x+h-y}\right) (\calH[f](x+h) - \calH[f](y)) (\calH[g](x+h) - \calH[g](y)) \ud y\\
            & \quad + \frac{1}{\pi} P.V.\int_{|x-y|>2h} \frac{1}{x-y} \bigg[ (\calH[f](x) - \calH[f](y))(\calH[g](x) - \calH[g](y)) \\
            & \qquad \qquad \qquad \qquad \qquad \qquad \qquad \qquad- (\calH[f](x+h) - \calH[f](y))(\calH[g](x+h) - \calH[g](y)) \bigg] \ud y\\
            & =: I_1+I_2+I_3+I_4.
        \end{split}
    \end{align*} 
    For $I_1$, we estimate
    \begin{align*}
        \begin{split}
            |I_1| 
            & \aleq \int_{|x-y|<2h} \left| \frac{(\calH[f](x) - \calH[f](y)) (\calH[g](x) - \calH[g](y))}{x-y} \right| \ud y\\
            & \aleq \nrm{\calH[f]}_{C^\vphi} \nrm{\calH[g]}_{C^{\vphi^\tht}} \int_{|x-y|<2h} \frac{\vphi(|x-y|)^{1+\tht}}{|x-y|} \ud y\\
            & \aleq \nrm{\calH[f]}_{C^\vphi} \nrm{\calH[g]}_{C^{\vphi^\tht}} \vphi(2h)\tld{\vphi^\tht}(2h) \\ 
            & \aleq \nrm{\calH[f]}_{C^\vphi} \nrm{\calH[g]}_{C^{\vphi^\tht}} \vphi(h) \\
            & \aleq \nrm{\calH[f]}_{C^\vphi} \nrm{\calH[g]}_{C^0}^{1-\tht} \nrm{\calH[g]}_{C^\vphi}^\tht \vphi(h),
        \end{split}
    \end{align*}
    where the last inequality comes from the interpolation inequality \eqref{eq:interpolation}. A similar argument yields the same estimates for $I_2$: 
    \begin{align*}
        |I_2| \aleq \nrm{\calH[f]}_{C^\vphi} \nrm{\calH[g]}_{C^0}^{1-\tht} \nrm{\calH[g]}_{C^\vphi}^\tht \vphi(h).
    \end{align*}
    For $I_3$, the bound on kernel \eqref{eq:bound on kernel} yields
    \begin{align*}
        \begin{split}
            |I_3| & \aleq \int_{|x-y|>2h} \left| \left(\frac{1}{x-y}-\frac{1}{x+h-y}\right) (\calH[f](x+h) - \calH[f](y)) (\calH[g](x+h) - \calH[g](y))\right| \ud y \\
            & \aleq \nrm{\calH[f]}_{C^\vphi} \nrm{\calH[g]}_{C^{\vphi^\tht}} \int_{|x-y|>2h} \frac{h}{|x-y|^2} \vphi(|x+h-y|)^{1+\tht} \ud y.
        \end{split}
    \end{align*}
    The integral on the right-hand side is bounded as follows.
    \begin{align*}
        \begin{split}
            \int_{|x-y|>2h} \frac{h}{|x-y|^2} &\vphi(|x+h-y|)^{1+\tht} \ud y \\
            & = h \left( \int_{2h< |x-y| < 1} \frac{\vphi(|x-y|)^{1+\tht}}{|x-y|^2} \ud y + \int_{|x-y| > 1} \frac{\vphi(|x-y|)^{1+\tht}}{|x-y|^2} \ud y \right) \\
            & \leq  h \left( \frac{\vphi(h)}{h}\int_{2h< |x-y| < 1} \frac{\vphi^\tht(|x-y|)}{|x-y|} \ud y + \int_{|x-y| > 1} \frac{\vphi(|x-y|)^{1+\tht}}{|x-y|^2} \ud y \right) \\
            & \aleq \vphi(h) \tld{\vphi^\tht}(1) + h\nrm{\vphi}_{L^\infty}^{1+\tht} \\
            & \aleq \vphi(h).
        \end{split}
    \end{align*}
    This bound leads to 
    \begin{align*}
        \begin{split}
            |I_3| & \aleq \nrm{\calH[f]}_{C^\vphi} \nrm{\calH[g]}_{C^{\vphi^\tht}} \vphi(h) \\
            & \aleq \nrm{\calH[f]}_{C^\vphi} \nrm{\calH[g]}_{C^0}^{1-\tht} \nrm{\calH[g]}_{C^\vphi}^\tht \vphi(h).
        \end{split}
    \end{align*}
    Finally, we can recast $I_4$ by using the symmetry:
    \begin{align*}
        \begin{split}
            I_4 
            & = \frac{1}{\pi} (\calH[f](x+h)-\calH[f](x))\,P.V.\int_{|x-y|>2h} \frac{\calH[g](y)}{x-y} \ud y \\
            & \quad + \frac{1}{\pi} (\calH[g](x+h)-\calH[g](x))\,P.V.\int_{|x-y|>2h} \frac{\calH[f](y)}{x-y} \ud y.
        \end{split}
    \end{align*}
    Then we obtain a bound for $|I_4|$ by using a lemma from Cotlar (see \cite{Torchinsky}, p.291):
    \begin{align*}
        |I_4| \aleq \left[\nrm{\calH[f]}_{C^\vphi}(\nrm{g}_{C^0} + \nrm{\calH[g]}_{C^0}) + \nrm{\calH[g]}_{C^\vphi}(\nrm{f}_{C^0} + \nrm{\calH[f]}_{C^0}) \right] \vphi(h).
    \end{align*}
    Combining the estimates for $I_1$ through $I_4$ yields \eqref{eq:Hfg nrm}. The simpler estimate \eqref{eq:Hfg nrm_simple} is then obtained by Young's inequality and $\nrm{f}_{C^0} \aleq \nrm{\calH[f]}_{C^\vphi}$.
\end{proof}

\subsection{Proof of Lemma \ref{Lem:Hf^n}} \label{Subsec:Hf^n}
\begin{proof}[Proof of Lemma \ref{Lem:Hf^n}]
    For notational convenience, set $\alp_0 := \vphi(\dlt)$. For $n\geq 2$, applying the product estimate \eqref{eq:Hfg nrm} on $f^n = f^{n-1} \cdot f$, we obtain a bound for $\nrm{\calH[f^n]}_{C^\vphi}$:
    \begin{align*}
        \begin{split}
            \nrm{\calH[f^n]}_{C^\vphi} 
            & \aleq \nrm{\calH[f^{n-1}]}_{C^\vphi} ( \nrm{\calH[f]}_{C^0}^{1-\tht}\nrm{\calH[f]}_{C^\vphi}^\tht + \nrm{\calH[f]}_{C^0} ) \\
            & \quad + \nrm{\calH[f]}_{C^\vphi} (\nrm{f^{n-1}}_{C^0} + \nrm{\calH[f^{n-1}]}_{C^0}).
        \end{split}
    \end{align*}
    Using the bounds on $f$ and the bound from \eqref{eq:Hf C^0}, the inequality simplifies to
    \begin{align*}
        \begin{split}
            \nrm{\calH[f^n]}_{C^\vphi}
            & \aleq \nrm{\calH[f^{n-1}]}_{C^\vphi} \times M(\alp_0^{1-2\tht+\tht^2} \cdot \alp_0^{-\tht^2} + \alp_0^{1-\tht}) \\
            & \quad + M\alp_0^{-\tht} ((M\alp_0)^{n-1} + (M\alp_0^{1-\tht})^{n-1}) \\
            & \aleq M \alp_0^{1-2\tht} \nrm{\calH[f^{n-1}]}_{C^\vphi} + \alp_0^{-1} (M\alp_0^{1-\tht})^n.
        \end{split}
    \end{align*}
    To make the dependence on constants explicit, we write
    \begin{align*}
        \nrm{\calH[f^n]}_{C^\vphi} \leq M C_{P,1} \alp_0^{1-2\tht} \nrm{\calH[f^{n-1}]}_{C^\vphi} + C_{P,2} \alp_0^{-1} (M\alp_0^{1-\tht})^n, 
    \end{align*}
    where the constants $C_{P,1}$ and $C_{P,2}$ depend only on $\vphi$. Let $a_n := \nrm{\calH[f^n]}_{C^\vphi}$ and $\alp_1 := C_{P,1}\alp_0^{1-2\tht}$. The recurrence inequality for $a_n$ is
    \begin{align*}
        a_n \leq (M\alp_1) a_{n-1} + C_{P,2} \alp_0^{-1} (M\alp_0^{1-\tht})^n. 
    \end{align*}
    Iterating this inequality, we get
    \begin{align*}
        \begin{split}
            a_n 
            & \leq (M\alp_1) a_{n-1} + C_{P,2} \alp_0^{-1} (M\alp_0^{1-\tht})^n \\
            & \leq (M\alp_1)^2 a_{n-2} + C_{P,2} \alp_0^{-1} (M\alp_0^{1-\tht})^n \left( 1+ \alp_1/\alp_0^{1-\tht}\right) \\
            & \qquad\qquad\qquad\qquad \qquad \quad\,\,\vdots \\
            & \leq (M\alp_1)^{n-1} a_1 + C_{P,2} \alp_0^{-1} (M\alp_0^{1-\tht})^n \left( \sum_{k=0}^{n-2} \left( \alp_1/\alp_0^{1-\tht} \right)^k \right).
        \end{split}
    \end{align*} 
    We bound the two terms separately. From the assumption on $a_1 = \nrm{\calH[f]}_{C^\vphi} \aleq M\alp_0^{-\tht}$, the first term is bounded by
    \begin{align*}
        \begin{split}
            (M\alp_1)^{n-1} a_1 
            & \aleq M^n \alp_0^{-\tht} \alp_1^{n-1} \\
            & \aleq \alp_0^{\tht-1} (M\alp_1)^n.
        \end{split}
    \end{align*}  
    For the second term, we choose $\dlt$ small enough such that $\alp_1 > 2\alp_0^{1-\tht}$, which is equivalent to $\vphi(\dlt)^\tht < C_{P,1}/2$. This ensures the ratio of the geometric series exceeds $2$, so that the sum is bounded by its largest term:
    \begin{align*}
        \begin{split}
            C_{P,2} \alp_0^{-1} (M\alp_0^{1-\tht})^n \left( \sum_{k=0}^{n-2} \left( \frac{\alp_1}{\alp_0^{1-\tht}} \right)^k \right) 
            & \leq C_{P,2} \alp_0^{-1} (M\alp_0^{1-\tht})^n \frac{(\alp_1/\alp_0^{1-\tht})^{n-1}}{\alp_1/\alp_0^{1-\tht}-1} \\
            & \aleq  \alp_0^{-1} M^n \alp_0^{1-\tht} \frac{\alp_1^{n-1}}{\alp_1/\alp_0^{1-\tht}-1} \\
            & \aleq \alp_0^{\tht-1} (M\alp_1)^n.
        \end{split}
    \end{align*}
    Combining the bounds for both terms, we obtain the desired estimate:
    \begin{align*}
        \nrm{\calH[f^n]}_{C^\vphi} \aleq \alp_0^{\tht-1} (M\alp_1)^n. 
    \end{align*}
\end{proof}

\subsection{Proof of Lemma \ref{Lem:Afg}} \label{Subsec:Afg}
\begin{proof}[Proof of Lemma \ref{Lem:Afg}]
    We decompose $\calA[f,g](x)$ by adding and subtracting terms. The integrand can be written as
    \begin{align*}
        f(y)g(x,y) = (f(y)-f(x))(g(x,y)-g(x,x)) + f(x)g(x,y) + (f(y)-f(x))g(x,x).
    \end{align*}
    Integrating each term against the kernel $\frac{1}{\pi(x-y)}$ yields the decomposition:
    \begin{align*}
        \begin{split}
            \calA[f,g](x) & = f(x) \left( \frac{1}{\pi} P.V.\int_\bbR \frac{g(x,y)}{x-y} \ud y \right) + g(x,x) \left( \frac{1}{\pi} P.V.\int_\bbR \frac{f(y)-f(x)}{x-y} \ud y \right) \\
            & \quad + \frac{1}{\pi} \int_\bbR \frac{(f(y)-f(x))(g(x,y)-g(x,x))}{x-y} \ud y \\
            & = f(x) \calH_2[g](x,x) + \calH[f](x)g(x,x) \\
            & \quad + \frac{1}{\pi} \int_\bbR \frac{(f(y)-f(x))(g(x,y)-g(x,x))}{x-y} \ud y.
        \end{split}
    \end{align*}
    The first and second terms are clearly $C^\vphi$. For the last term, we can obtain the $C^0$ bounds by using the compact support condition. Its $C^\vphi$ norm can be estimated by \eqref{eq:prod_cornrm}:
    \begin{align*}
        \left\Vert \int_\bbR \frac{f(y)-f(x)}{x-y} (g(x,y)-g(x,x)) \ud y \right\Vert_{C^\vphi} \aleq \nrm{f}_{C^\vphi}\nrm{g}_{C^\vphi}.
    \end{align*}
    Combining the norm estimates for all three terms, we obtain the desired bound: 
    \begin{align*}
        \nrm{\calA[f,g]}_{C^\vphi} 
        & \aleq \nrm{f \cdot \calH_2[g]}_{C^\vphi} + \nrm{\calH[f] \cdot g}_{C^\vphi} + \nrm{f}_{C^\vphi}\nrm{g}_{C^\vphi} \\
        & \aleq \nrm{f}_{C^\vphi} \nrm{\calH_2[g]}_{C^\vphi} + \nrm{\calH[f]}_{C^\vphi} \nrm{g}_{C^\vphi} + \nrm{f}_{C^\vphi} \nrm{g}_{C^\vphi}.
    \end{align*}
\end{proof}

\subsection{Proof of Lemma \ref{Lem:T}} \label{Subsec:T}
\begin{proof} [Proof of Lemma \ref{Lem:T}]
    First, note that $\calT_j^{n_1,\cdots,n_l}[f]$ is $C^\vphi$ as it is of the form $\calA[f',\Pi\calQ[f]]$. We will only prove the case $j \neq 0$, as the case $j=0$ can be handled similarly.
    
    By Fubini's theorem, we can recast $\calH[\calT_j^{n_1,\dots,n_l}[f]]$ as
    \begin{align*}
        \begin{split}
            \calH[\calT_j^{n_1,\cdots,n_l}[f]](x) 
            & = P.V.\int_\bbR \frac{1}{x-y} P.V.\int_\bbR \frac{f'_j(z)}{y-z} \Pi\calQ [f]^{n_1,\cdots,n_l}(y,z) \ud z\ud y \\
            & = \lim_{\veps \to 0} \iint_{A(x,\veps)} \frac{1}{x-y}\frac{f'_j(z)}{y-z} \Pi\calQ [f]^{n_1,\cdots,n_l}(y,z) \ud z\ud y,
        \end{split}
    \end{align*}
    where $A(x,\veps) = \left\{(y,z) \in \bbR \big| \,|x-y|,|y-z|,|z-x|>\veps\right\}$. Using the partial fraction decomposition on $\frac{1}{(x-y)(y-z)}$, we decompose the integral into two parts:
    \begin{align*}
        \begin{split}
            \calH[\calT_j^{n_1,\cdots,n_l}[f]](x) 
            & = \lim_{\veps \to 0} \iint_{A(x,\veps)} \frac{f'_j (z)}{x-z} \frac{1}{y-z} \Pi\calQ [f]^{n_1,\cdots,n_l}(y,z) \ud z\ud y\\
            & \quad - \lim_{\veps \to 0} \iint_{A(x,\veps)} \frac{f'_j (z)}{x-z} \frac{1}{y-x} \Pi\calQ [f]^{n_1,\cdots,n_l}(y,z) \ud z\ud y\\
            & =: J_1 + J_2.
        \end{split}
    \end{align*} 

    \noindent\textbf{STEP 1: Calculation of $J_1$}

    Applying Fubini's theorem to $J_1$, we have
    \begin{align*}
        \begin{split}
            J_1 & = \lim_{\veps \to 0} \iint_{A(x,\veps)} \frac{f'_j (z)}{x-z} \frac{1}{y-z} \Pi\calQ[f]^{n_1,\cdots,n_l}(y,z) \ud z\ud y \\
            & = \lim_{\veps \to 0} \int_{|x-z|>\veps} \frac{f'_j (z)}{x-z} 
            \int_{|y-z|>\veps} \frac{1}{y-z} \Pi\calQ[f]^{n_1,\cdots,n_l}(y,z) \ud y\ud z.
        \end{split}
    \end{align*}
    We will use the identity for the derivative of the product of difference quotients:
    \begin{align} \label{eq:dv of PiQ}
        \begin{split}
            \frac{1}{y-z} \Pi\calQ[f]^{n_1,\cdots,n_l} (y,z) 
            & = \sum_{k=1}^l \frac{n_k}{n} \frac{f'_k (y)}{y-z} \Pi\calQ[f]^{n_1,\cdots, n_k-1, \cdots ,n_l} (y,z) \\
            & \quad - \frac{1}{n} \frac{\ud}{\ud y} \left[ \Pi\calQ[f]^{n_1,\cdots,n_l} \right] (y,z),
        \end{split}
    \end{align}
    where $n = n_1 + \cdots + n_l$. Applying \eqref{eq:dv of PiQ} to $J_1$ and integrating by parts leads to 
    \begin{align*}
        \begin{split}
            J_1
            & = \sum_{k=1}^l \frac{n_k}{n} \lim_{\veps \to 0} \int_{|x-z|>\veps}\frac{f'_j(z)}{x-z} \int_{|y-z|>\veps}\frac{f'_k (y)}{y-z} \Pi\calQ[f]^{n_1,\cdots, n_k-1, \cdots ,n_l} (y,z) \ud y \ud z\\
            & \quad - \frac{1}{n} \lim_{\veps \to 0} \int_{|x-z|>\veps} \frac{f'_j (z)}{x-z}
            \int_{|y-z|>\veps} \frac{\ud}{\ud y} \left[ \Pi\calQ[f]^{n_1,\cdots,n_l} \right](y,z) \ud y\ud z \\
            & = \sum_{k=1}^l \frac{n_k}{n} \calH[f'_j \cdot \calT_k^{n_1,\cdots,n_k-1,\cdots,n_l}[f]] \\
            & \quad - \frac{1}{n} \lim_{\veps \to 0} \underbrace{\int_{|x-z|>\veps} \frac{f'_j(z)}{x-z}[\Pi\calQ[f]^{n_1,\cdots,n_l}(z+\veps) - \Pi\calQ[f]^{n_1,\cdots,n_l}(z-\veps)]}_{ = : (\diamondsuit)}.
        \end{split}
    \end{align*}
    We now claim that $(\diamondsuit)$ vanishes as $\veps \to 0$. We decompose it into the two terms
    \begin{align*}
        \begin{split}
            (\diamondsuit)
            & = \int_{|x-z|>\veps} \frac{f'_j(z) - f'_j(x)}{x-z}[\Pi\calQ[f]^{n_1,\cdots,n_l}(z+\veps) - \Pi\calQ[f]^{n_1,\cdots,n_l}(z-\veps)] \ud z\\
            & \quad + f'(x)\int_{|x-z|>\veps} \frac{1}{x-z}[\Pi\calQ[f]^{n_1,\cdots,n_l}(z+\veps) - \Pi\calQ[f]^{n_1,\cdots,n_l}(z-\veps)] \ud z \\
            & = : (\diamondsuit_1) + (\diamondsuit_2).
        \end{split}
    \end{align*}
    For the first term, the compact support condition of $f$ and yields
    \begin{align*}
        \begin{split}
            |(\diamondsuit_1)| 
            & \leq \int_{\veps<|x-z|<C} \left|\frac{f'_j(z) - f'_j(x)}{x-z}[\Pi\calQ[f]^{n_1,\cdots,n_l}(z+\veps) - \Pi\calQ[f]^{n_1,\cdots,n_l}(z-\veps)] \right| \ud z \\
            & \leq \nrm{f'_j}_{C^\vphi} \nrm{\Pi\calQ[f]^{n_1,\cdots,n_l}}_{C^\vphi} \vphi(\veps)\int_{\veps<|x-z|<C} \frac{\vphi(|x-z|)}{|x-z|} \ud z \to 0.
        \end{split}
    \end{align*}
    For the second term, we apply \eqref{A3}: $\displaystyle \vphi(\veps) \log(\veps) \to 0$ as $\veps \to 0$, to obtain 
    \begin{align*}
        \begin{split}
            |(\diamondsuit_2)|
            & \leq \nrm{\Pi\calQ[f]^{n_1,\cdots,n_l}}_{C^\vphi} \vphi(\veps) \int_{|x-z|>\veps} \frac{1}{|x-z|}\ud z \\
            & \leq \nrm{\Pi\calQ[f]^{n_1,\cdots,n_l}}_{C^\vphi} \vphi(\veps) \left(\log\frac{C}{\veps}\right) \to 0.
        \end{split}
    \end{align*}

    \noindent\textbf{STEP 2: Calculation of $J_2$} 
    
    Consider the following decomposition of $J_2$: 
    \begin{align*}
        \begin{split}
            J_2 
            & = \lim_{\veps \to 0} \int_{|x-z|>\veps} \frac{f'_j (x)}{x-z} \int_{|y-x|>\veps} \frac{1}{x-y} \Pi\calQ[f]^{n_1,\cdots,n_l}(y,z) \ud y\ud z \\
            & \quad + \lim_{\veps \to 0} \int_{|x-z|>\veps} \frac{f'_j (z)}{x-z} \int_{|y-x|>\veps} \frac{1}{x-y} \Pi\calQ[f]^{n_1,\cdots,n_l}(y,x) \ud y\ud z \\
            & \quad + Remainder \\
            & =: J_{2,1} + J_{2,2} + J_{\mathrm{rem}}.
        \end{split}
    \end{align*}
    Clearly, $J_{2,2} = \calH[f'_j] \calT_0^{n_1,\cdots,n_l}[f]$. For $J_{2,1}$, we apply the partial fraction decomposition to obtain
    \begin{align*}
        \begin{split}
            & \lim_{\veps \to 0} \int_{|x-z|>\veps} \frac{1}{x-z} \int_{|y-z|>\veps} \frac{1}{x-y} \Pi\calQ[f]^{n_1,\cdots,n_l}(y,z) \ud y\ud z \\
            & = \lim_{\veps \to 0} \iint_{A(x,\veps)} \frac{1}{x-z} \frac{1}{y-z} \Pi\calQ[f]^{n_1,\cdots,n_l}(y,z) \ud y\ud z \\
            & \quad - \lim_{\veps \to 0} \iint_{A(x,\veps)} \ \frac{1}{x-y} \frac{1}{y-z} \Pi\calQ[f]^{n_1,\cdots,n_l}(y,z) \ud y\ud z.
        \end{split}
    \end{align*}
    By applying the change of variables $(y,z) \to (z,y)$ on the first term, we have
    \begin{align*}
        J_{2,1} = 
        & = -2 \lim_{\veps \to 0} \iint_{A(x,\veps)} \frac{1}{x-y} \frac{1}{y-z} \Pi\calQ[f]^{n_1,\cdots,n_l}(y,z) \ud y\ud z.
    \end{align*}
    Applying integration by parts and \eqref{eq:dv of PiQ} on the above formula, we conclude that 
    \begin{align*}
        J_{2,1} = -2 f'_j(x) \sum_{k=1}^l \frac{n_k}{n} \calH[\calT_k^{n_1,\cdots,n_k-1,\cdots,n_l}[f]].
    \end{align*}
    Note that $J_{\mathrm{rem}}$ is given by
    \begin{align*}
        J_{\mathrm{rem}} = P.V.\int_\bbR \frac{f'_j(x) - f'_j (z)}{x-z} (\calN(x,x) - \calN(x,z)) \ud z,
    \end{align*}
    where
    \begin{align*}
        \begin{split}
            \calN(x,z) & = P.V.\int_\bbR \frac{1}{x-y} \Pi\calQ[f]^{n_1,\cdots,n_l}(y,z) \ud y \\
            & = \calH_1[\Pi\calQ[f]^{n_1,\cdots,n_l}] (x,z).
        \end{split}
    \end{align*}
    By \eqref{eq:prod_cornrm}, to guarantee that $J_{\mathrm{rem}} \in C^\vphi$, it suffices to check that $\calN(x,z)$ satisfies $C^\vphi$. Since each $\calQ[f_l] \in C^\vphi$ and $\calH_1[\calQ[f_l]] = \calQ[\calH[f_l]] \in C^\vphi$, we conclude that $\calN = \calH_1[\Pi \calQ [f]^{n_1,\cdots,n_l}]$ is in $C^\vphi$.
\end{proof}

\subsection{Proof of Lemma \ref{Lem:T2}} \label{Subsec:T2}
\begin{proof} [Proof of Lemma \ref{Lem:T2}]
    In this proof, we also denote $\vphi(\dlt)$ as $\alp_0$. We further assume that $\dlt$ is small enough so that $\vphi(\dlt)^\tht < C_{P,1}/2$ and thus the product estimate \eqref{eq:Hf^n nrm} holds for $\calQ[f_j]$:
    \begin{align*}
        \nrm{\calH[\calQ[f_j]^n]}_{C^\vphi} \aleq \vphi(\dlt)^{\tht-1} (M\alp_1)^n,
    \end{align*}
    where $\alp_1 = C_{P,1} \alp_0^{1-2\tht}$. 
    
    \noindent\textbf{STEP 1: Estimates of $\calT_j^{n_1,n_2}[f]$} 
    
    We begin by establishing the estimates for $\calT_0^{n_1,n_2}[f]$, which is the simplest case. Observe that $\calT_0^{n_1,n_2}[f](x) = \calH_2[\Pi\calQ[f]^{n_1,n_2}](x,x)$. Since $f$ and $\calQ[f]$ are compactly supported, we apply \eqref{eq:Hf C^0} and \eqref{eq:Hf^n nrm} to obtain the $C^0$ and $C^\vphi$ bounds respectively:
    \begin{align*}
        \begin{split}
            \nrm{\calT_0^{n_1,n_2}[f]}_{C^0} 
            & \aleq \nrm{\Pi\calQ[f]^{n_1,n_2}}_{C^0}^{1-\tht} \nrm{\Pi\calQ[f]^{n_1,n_2}}_{C^\varphi}^\tht \\
            & \aleq (\nrm{f'}_{C^0}^{n_1+n_2})^{1-\tht} (\nrm{f'}_{C^0}^{n_1+n_2-1}\nrm{f'}_{C^\varphi})^\tht \\
            & \aleq M^{n_1+n_2} \cdot \alp_0^{(n_1+n_2)(1-\tht)} \cdot \alp_0^{(n_1+n_2-1)\tht} \\
            & = \alp_0^{-\tht} (M\alp_0)^{n_1+n_2},
        \end{split}
    \end{align*}
    and
    \begin{align*}
        \begin{split}
            \nrm{\calT_0^{n_1,n_2}[f]}_{C^\vphi} 
            & = \nrm{\calH_1 [\Pi\calQ[f]^{n_1,n_2}]}_{C^\vphi} \\
            & \aleq \nrm{\calH_1 [\calQ[f_1]^{n_1}]}_{C^\vphi} \nrm{\calH_1 [\calQ[f_2]^{n_2}]}_{C^\vphi} \\
            & \aleq \alp_0^{2(\tht-1)} (M\alp_1)^{n_1+n_2}.
        \end{split}
    \end{align*}
    We now estimate the norms of $\calT_j^{n_1,n_2}[f]$ for $j = 1,2$. Since $\calT$ is of the form $\calA[f',\Pi\calQ[f]]$, we apply the formula \eqref{eq:Afg nrm}: 
    \begin{align*}
        \begin{split}
            \nrm{\calT_j^{n_1,n_2}[f]}_{C^0} 
            & \aleq \nrm{f'_j}_{C^0} \nrm{\calH_2[\Pi\calQ[f]^{n_1,n_2}]}_{C^0} + \nrm{\calH[f'_j]}_{C^0} \nrm{\Pi\calQ[f]^{n_1,n_2}}_{C^0} \\
            & \quad + \nrm{f'_j}_{C^0} \nrm{\Pi\calQ[f]^{n_1,n_2}}_{C^0} \\ 
            & \aleq M\alp_0 \cdot \alp_0^{-\tht} (M\alp_0)^{n_1+n_2} + M\alp_0^{1-\tht} \cdot (M\alp_0)^{n_1+n_2} + (M\alp_0)^{n_1+n_2+1} \\
            & \aleq \alp_0^{-\tht} (M\alp_0)^{n_1+n_2+1},
        \end{split}
    \end{align*}
    and
    \begin{align*}
        \begin{split}
            \nrm{\calT_j^{n_1,n_2}[f]}_{C^\vphi} 
            & \aleq \nrm{f'_j}_{C^\vphi} \nrm{\calH_2[\Pi\calQ[f]^{n_1,n_2}]}_{C^\vphi} + \nrm{\calH[f'_j]}_{C^\vphi} \nrm{\Pi\calQ[f]^{n_1,n_2}}_{C^\vphi} \\
            & \quad + \nrm{f'_j}_{C^\vphi} \nrm{\Pi\calQ[f]^{n_1,n_2}}_{C^\vphi} \\ 
            & \aleq M\cdot\alp_0^{2(\tht-1)} (M\alp_1)^{n_1+n_2} + M\alp_0^{-\tht} \cdot \alp_0^{-1}(M\alp_0)^{n_1+n_2} + M\cdot\alp_0^{-1}(M\alp_0)^{n_1+n_2} \\
            & \aleq \alp_0^{4\tht-3} (M\alp_1)^{n_1+n_2+1}.
        \end{split} 
    \end{align*}

    \noindent\textbf{STEP 2: $C^0$ estimates of $\calH[\calT_j^{n_1,n_2}[f]]$} 

    Since $\calT_j^{n_1,n_2}[f]$ is not compactly supported, we cannot apply the interpolation inequality \eqref{eq:Hf C^0} directly to obtain the bounds on $\nrm{\calH[\calT_j^{n_1,n_2}]}_{C^0}$. Hence we obtain its bound inductively using the formulas \eqref{eq:T formula j=0} and \eqref{eq:T formual j!=0}. For $n \geq 1$, define the sequences $b_n$ and $c_n$ as 
    \begin{align*}  
        \begin{split}
            b_n := \max_{n_1+n_2 = n} \nrm{\calH[\calT_0^{n_1,n_2}]}_{C^0} \\
            c_n := \max_{\substack{n_1+n_2+1 =n \\ j=1,2}} \nrm{\calH[\calT_j^{n_1,n_2}]}_{C^0}.
        \end{split}
    \end{align*}
    Then from the formula \eqref{eq:T formula j=0}, we obtain that $b_n \leq c_n$. For the bounds on $c_n$, we first recall the formula for $\calH[\calT_j^{n_1,n_2}]$ for $n_1,n_2 \geq 1$:
    \begin{align*}
        \begin{split}
            \calH[\calT_j^{n_1,n_2}[f]] 
            & = \underbrace{\frac{n_1}{n} \calH[f'_j \cdot \calT_1^{n_1-1,n_2}[f]] + \frac{n_2}{n} \calH[f'_j \cdot \calT_2^{n_1,n_2-1}[f]]}_{=:(\clubsuit_1)} \\
            & \quad \underbrace{-2f'_j \left[ \frac{n_1}{n} \calH[\calT_1^{n_1-1,n_2}[f]] + \frac{n_2}{n} \calH[\calT_1^{n_1,n_2-1}[f]] \right]}_{=:(\clubsuit_2)} + \underbrace{\calH[f'_j]\calT_0^{n_1,n_2}[f]}_{=:(\clubsuit_3)} \\
            & \quad + J_{\mathrm{rem}},
        \end{split}   
    \end{align*} 
    where
    \begin{align*}
        \begin{cases}
            & J_{\mathrm{rem}} = P.V.\int_\bbR \frac{f'_j(x) - f'_j (z)}{x-z} (\calN(x,x) - \calN(x,z)) \ud z \\
            & \calN(x,z) = \calH_1[\Pi\calQ[f]^{n_1,n_2}](x,z).
        \end{cases}
    \end{align*}
    Since the cases of $n_1 = 0$ and $n_2=0$ can be treated similarly, we will present the bounds for $n_1,n_2 \geq 1$. From the formula, we observe that the terms $f'_j \cdot \calT_k^{\cdots,n_k-1,\cdots}[f]$ are compactly supported. Hence \eqref{eq:Hf C^0} is applicable to $(\clubsuit_1)$ and it yields
    \begin{align*}
        \begin{split}
            \nrm{(\clubsuit_1)}_{C^0} 
            & \aleq \nrm{f'_j \cdot \calT_k^{\cdots,n_k-1\cdots} [f]}_{C^0}^{1-\tht} \nrm{f'_j \cdot \calT_k^{\cdots,n_k-1\cdots} [f]}_{C^\vphi}^\tht \\
            & \aleq [ M\alp_0 \cdot \alp_0^{-\tht}(M\alp_0)^{n_1+n_2} ]^{1-\tht} [ M \cdot \alp_0^{4\tht-3} (M\alp_1)^{n_1+n_2} ]^\tht \\
            & \aleq \alp_0^{\tht(7\tht-5)} (M\alp_0^{1-\tht}\alp_1^\tht)^{n_1+n_2+1} \\
            & \aleq \alp_0^{\tht(7\tht-5)} (M\alp_2)^{n_1+n_2+1},
        \end{split}
    \end{align*}
    where $\alp_2 := \vphi(\dlt)^{1-\tht}\alp_1^\tht$. We note that for $\dlt$ small enough, $\alp_2$ is of order $\vphi(\dlt)^{1-2\tht^2}$ and $\alp_0<\alp_2<\alp_1$. For the other terms, we can verify that
    \begin{align*}
        \begin{split}
            \nrm{(\clubsuit_2)}_{C^0} + \nrm{(\clubsuit_3)}_{C^0}  
            & \aleq M\alp_0 \cdot c_{n_1+n_2} + M\alp_0^{1-\tht} \cdot \alp_0^{-\tht}(M\alp_0)^{n_1+n_2} \\
            & \aleq M\alp_0 \cdot c_{n_1+n_2} + \alp_0^{-2\tht} (M\alp_0)^{n_1+n_2+1}.
        \end{split}
    \end{align*}
    For the remainder term, \eqref{eq:Hf C^0} leads to 
    \begin{align*}
        \begin{split}
            \nrm{J_{\mathrm{rem}}}_{C^0} 
            & \aleq \nrm{f'_j}_{C^0}^{1-\tht}\nrm{f'_j}_{C^\vphi}^\tht\nrm{\calN}_{C^0}^{1-\tht}\nrm{\calN}_{C^\vphi}^\tht \\
            & \aleq (M\alp_0)^{1-\tht} \cdot M^\tht \cdot (\alp_0^{-\tht}(M\alp_0)^{n_1+n_2})^{1-\tht} (\alp_0^{2(\tht-1)}(M\alp_1)^{n_1+n_2})^\tht \\ 
            & \aleq \alp_0^{\tht(5\tht-4)} (M\alp_2)^{n_1+n_2+1}.
        \end{split} 
    \end{align*}
    Combining these estimates, we obtain the following recurrence inequality:
    \begin{align*}
        c_{n_1+n_2+1} \aleq M\alp_0 \cdot c_{n_1+n_2} + (\alp_0^{\tht(5\tht-4)}+\alp_0^{\tht(7\tht-5)}) (M\alp_2)^{n_1+n_2+1} + \alp_0^{-2\tht} (M\alp_0)^{n_1+n_2+1}.
    \end{align*}
    Since $0 > \tht(5\tht-4) > \tht(7\tht-5)$ and $\alp_2>\alp_0$, we can eliminate the two low-order terms. Finally, we have 
    \begin{align*}
        c_{n} \leq C_{T_2,1} M \alp_0 \cdot c_{n-1} + C_{T_2,2} \alp_0^{\tht(7\tht-5)} (M\alp_2)^{n},
    \end{align*} 
    for any $n > 1$, where $C_{T_2,1}$ and $C_{T_2,2}$ are constants that depend only on $\vphi$. By using a method similar to that in the proof of Lemma \ref{Lem:Hf^n}, we obtain the bounds on $c_n$ inductively.
    \begin{align*}
        \begin{split}
            c_n 
            & \leq (M C_{T_2,1}\alp_0) c_{n-1} + C_{T_2,2} \alp_0^{\tht(7\tht-5)} (M\alp_2)^{n} \\
            & \leq (M C_{T_2,1}\alp_0)^{n-1}c_1 + C_{T_2,2} \alp_0^{\tht(7\tht-5)} (M\alp_2)^n \left( \sum_{k=0}^{n-2} \left( \frac{C_{T_2,1\alp_0}}{\alp_2} \right)^k \right)
        \end{split}
    \end{align*}
    We further assume that $\dlt$ is small enough so that $2C_{T_{2,1}}\alp_0<\alp_2$. Then since the $a_2\ageq a_0$, we have That
    \begin{align*}
        c_n \leq C_{T_2} \alp_0^{\tht(7\tht-5)} (M\alp_2)^n
    \end{align*}

    \noindent\textbf{STEP 3: $C^\vphi$ estimate of $\calH[\calT_j^{n_1,n_2}[f]]$} 
    
    In this step, we also assuem that $n_1,n_2 \geq 1$. Let $\displaystyle d_n = \max_{j=1,2}\nrm{\calH[\calT_j^{n_1,n_2}[f]]}_{C^\vphi}$. For $(\clubsuit_1)$, we apply \eqref{eq:Hfg nrm} to obtain 
    \begin{align*}
        \begin{split}
            \nrm{(\clubsuit_1)}_{C^\vphi} 
            & \aleq d_{n_1+n_2} (\nrm{\calH[f'_j]}_{C^0}^{1-\tht} \nrm{\calH[f'_j]}_{C^\vphi}^\tht + \nrm{\calH[f'_j]}_{C^0}) + \nrm{\calH[f'_j]}_{C^\vphi}( \nrm{\calT_{\cdots}[f]}_{C^0} \\
            & \quad + \nrm{\calH[\calT_{\cdots}[f]]}_{C^0}) \\
            & \aleq d_{n_1+n_2}(M\alp_0^{1-2\tht}+M\alp_0^{1-\tht}) + M\alp_0^{-\tht} ( \alp_0^{-\tht} (M\alp_0)^{n_1+n_2} + \alp_0^{\tht(7\tht-5)} (M\alp_2)^{n_1+n_2}) \\
            & \aleq M\alp_0^{1-2\tht} \cdot d_{n_1+n_2} + \alp_0^{\tht(7\tht-5)} (M\alp_2)^{n_1+n_2} \\
            & \aleq M\alp_1 \cdot d_{n_1+n_2} + \alp_0^{\tht(7\tht-5)} (M\alp_2)^{n_1+n_2}.
        \end{split}
    \end{align*}
    For the second term, using the algebra properties of $C^\vphi$ we obtain 
    \begin{align*}
        \begin{split}
            \nrm{(\clubsuit_2)}_{C^\vphi} 
            & \leq \nrm{f'_j}_{C^0} \cdot d_{n_1+n_2} + \nrm{f'_j}_{C^\vphi} \nrm{\calH[\calT_{\cdots}[f]]}_{C^0} \\
            & \aleq M\alp_0 \cdot d_{n_1+n_2} + M \cdot \alp_0^{\tht(7\tht-5)} (M\alp_2)^{n_1+n_2}.
        \end{split}
    \end{align*}
    For the third term, we estimate
    \begin{align*}
        \begin{split}
            \nrm{(\clubsuit_3)}_{C^\vphi} 
            & \leq \nrm{\calH[f'_j]}_{C^\vphi} \nrm{\calT_0^{n_1,n_2}[f]}_{C^\vphi} \\
            & \aleq M\alp_0^{-\tht} \cdot \alp_0^{2\tht-2} (M\alp_1)^{n_1+n_2}.
        \end{split}
    \end{align*}
    For the remainder term, \eqref{eq:prod_cornrm} yields
    \begin{align*}
        \begin{split}
            \nrm{J_{\mathrm{rem}}}_{C^\vphi} 
            & \aleq \nrm{f'_j}_{C^\vphi} \nrm{\calN}_{C^\vphi} \\
            & \aleq M \cdot \alp_0^{2\tht-2} (M\alp_1)^{n_1+n_2}.
        \end{split}
    \end{align*}
    Combining these estimates and after eliminating the low-order terms, we have the following recurrence inequality for $d_n$:
    \begin{align*}
        d_n \aleq M\alp_1 \cdot d_{n-1} + \alp_0^{C(\tht)} (M\alp_1)^n.
    \end{align*}
    By applying a similar argument, we finally obtain
    \begin{align*}
        d_n \leq C_{T_2} \alp_0^{C(\tht)} (M\alp_1)^n.
    \end{align*}
    for some $C(\tht)$.
\end{proof}

\subsection{Proof of Lemma \ref{Lem:T6}} \label{Subsec:T6}
\begin{proof} [Proof of Lemma \ref{Lem:T6}]
    First, we take $\dlt <1$ small enough such that 
    \begin{align*}
        \nrm{\calH[\calQ[f_k]^n]}_{C^\vphi} \leq (M\alpha_1)^n .
    \end{align*}

    \noindent \textbf{CASE 1: $j\in\{3,4,5,6\}$, $n_1 = n_2 = 0$} 
    
    One can use a similar method to the proof of Lemma \ref{Lem:T2}. Note that this case coincides with the case where $l=4$.
    \begin{align}
        \nrm{\calH[\calT_j^{0,0,n_3,\cdots,n_6}[f]]}_{C^\vphi} \leq (M\alpha_1)^n 
    \end{align}

    \noindent \textbf{CASE 2: $j\in\{1,2\}$, $n_1 = n_2 = 0$}

    In this step, denote $n=n_3+\cdots+n_6$. From the formula \eqref{eq:T formual j!=0}, we have that 
    \begin{align*}
        \begin{split}
            \calH[\calT_j^{0,0,n_3,\cdots,n_6}[f]]
            & = \sum_{k=3}^6 \frac{n_k}{n} \calH[f'_j \cdot \calT_k^{0,0,\cdots,n_k-1,\cdots}[f]]\\
            & \quad - 2f'_j \sum_{k=3}^6 \frac{n_k}{n} \calH[\calT_k^{0,0,\cdots,n_k-1,\cdots}[f]] + \calH[f'_j] \calT_0^{0,0,n_3,\cdots,n_6}[f] \\
            & \quad + J_{\mathrm{rem}}.
        \end{split}
    \end{align*}
    For the first term, use \eqref{eq:Hfg nrm_simple}
    \begin{align*}
        \begin{split}
            \nrm{(\clubsuit_1)}_{C^\vphi}
            & \leq \nrm{\calH[f'_j]}_{C^\vphi} \left( \sum_{k=3}^6 \frac{n_k}{n} \nrm{\calH[\calT_k^{0,0,\cdots,n_k-1,\cdots}[f]]}_{C^\vphi} \right) \\
            & \aleq \rho \alp_0^{-\tht} \cdot \alp_0^{C(\tht)} (M\alp_1)^n \\
            & \aleq \rho \alp_0^{C(\tht)} (M\alp_1)^{n}
        \end{split}
    \end{align*}
    and for the second term, use the algebra properties of $C^\vphi$ norm and obtain that
    \begin{align*}
        \begin{split}
            \nrm{(\clubsuit_2)}_{C^\vphi}
            & \aleq \nrm{f'_j}_{C^\vphi} \left( \sum_{k=3}^6 \frac{n_k}{n} \nrm{\calH[\calT_k^{0,0,\cdots,n_k-1,\cdots}[f]]}_{C^\vphi} \right) \\
            & \aleq \rho \cdot \alp_0^{C(\tht)} (M\alp_1)^n \\
            & \aleq \rho \alp_0^{C(\tht)} (M\alp_1)^{n}
        \end{split}
    \end{align*}
    For the third term, we have that
    \begin{align*}
        \begin{split}
            \nrm{(\clubsuit_3)}_{C^\vphi}
            & \aleq \rho \alp_0^{-\tht} \cdot \alp_0^{C(\tht)} (M\alp_1)^{n}\\
            & \aleq \rho \alp_0^{C(\tht)} (M\alp_1)^{n}.
        \end{split}
    \end{align*}
    For $J_{\mathrm{rem}}$,
    \begin{align*}
        \begin{split}
            \nrm{J_{\mathrm{rem}}}_{C^\vphi}
            & \aleq \nrm{f'_j}_{C^\vphi} \nrm{\calN}_{C^\vphi} \\
            & \aleq C(M) \rho \alpha_1^n.
        \end{split}
    \end{align*}
    To sum up, we get
    \begin{align*}
        \nrm{\calH[\calT_j^{0,0,n_3,\cdots,n_6}[f]]}_{C^\vphi} \leq \rho \alp_0^{C(\tht)} (M\alp_1)^{n}.
    \end{align*}

    \noindent \textbf{CASE 3: $j\in\{3,4,5,6\}$ and $(n_1,n_2) \in \{(1,0),(0,1)\}$}
    Without loss of generality, let $n_1 = 1$ and $n_2 = 0$. Also, we denote $n = n_3+\cdots+n_6$. We first check that
    \begin{align*}
        \begin{split}
            \nrm{\calT_j^{1,0,n_3,\cdots,n_6}[f]}_{C^\vphi}
            & \aleq \nrm{\calH[f'_j]}_{C^\vphi} \nrm{\Pi\calQ[f]^{1,0,n_3,\cdots,n_6}}_{C^\vphi} \\
            & \quad + \nrm{f'_j}_{C^\vphi} \nrm{\calH[\Pi\calQ[f]^{1,0,n_3,\cdots,n_6}]}_{C^\vphi} \\
            & \quad + \nrm{f'_j}_{C^\vphi} \nrm{\Pi\calQ[f]^{1,0,n_3,\cdots,n_6}}_{C^\vphi} \\
            & \aleq \nrm{\calH[f'_j]}_{C^\vphi} \nrm{\calQ[f_1]}_{C^\vphi} \nrm{\Pi\calQ[f]^{0,0,n_3,\cdots,n_6}}_{C^\vphi} \\
            & \quad + \nrm{f'_j}_{C^\vphi} \nrm{\calH[\calQ[f_1]]}_{C^\vphi} \nrm{\calH[\Pi\calQ[f]^{0,0,n_3,\cdots,n_6}]}_{C^\vphi} \\
            & \quad + \nrm{f'_j}_{C^\vphi} \nrm{\calQ[f_1]}_{C^\vphi} \nrm{\Pi\calQ[f]^{0,0,n_3,\cdots,n_6}}_{C^\vphi} \\
            & \aleq \rho \alp_0^{C(\tht)} (M\alp_1)^{n+1}
        \end{split}
    \end{align*}
    By using the formula \eqref{eq:T formual j!=0}, we have the following expression:
    \begin{align*}
        \begin{split}
            \calH[\calT_j^{1,0,n_3,\cdots,n_6}[f]] 
            & = \underbrace{\frac{1}{n+1}\calH[f'_j \cdot \calT_1^{0,0,n_3,\cdots,n_6}[f]]}_{=:(\spadesuit_1)} 
            + \underbrace{\sum_{k=3}^6 \frac{n_k}{n+1} \calH[f'_j \cdot \calT_k^{1,0,\cdots,n_k-1,\cdots}[f]]}_{=:(\spadesuit_2)} \\
            & \quad - \underbrace{\frac{2}{n+1} f'_j \calH[\calT_1^{0,0,n_3,\cdots,n_6}[f]]}_{=:(\spadesuit_3)} 
            - \underbrace{2f'_j  \sum_{k=3}^6 \frac{n_k}{n+1} \calH[\calT_k^{1,0,\cdots,n_k-1,\cdots}[f]]}_{=:(\spadesuit_4)} \\
            & \quad + \underbrace{\calH[f'_j] \calT_0^{1,0,n_3,\cdots,n_6}[f]}_{=:(\spadesuit_5)} + J_{\mathrm{rem}}
        \end{split}
    \end{align*}
    Define the sequence $a_n$ as 
    \begin{align}
        a_n = \max_{\substack{j = 3,4,5,6 \\ 2 + n_3+\cdots+n_6 = n} } \nrm{\calH[\calT_j^{1,0,n_3,\cdots,n_6}][f]}_{C^\vphi} \qquad \text{for $n\geq 2$}
    \end{align}
    We can obtain the bounds for the first term by using the algebra properties of $C^\vphi$.  
    \begin{align*}
        \begin{split}
            \nrm{(\spadesuit_1)}_{C^\vphi}
            & \aleq \frac{1}{n+1} \nrm{\calH[f'_j]}_{C^\vphi} \nrm{\calH[\calT_1^{0,0,n_3,\cdots,n_6}]}_{C\vphi} \\
            & \aleq \frac{1}{n+1} \rho \alp_0^{-\tht} \cdot \alp_0^{C(\tht)} (M\alp_1)^{n+1} \\
            & \aleq \rho \alp_0^{C(\tht)} (M\alp_1)^{n+1}
        \end{split}
    \end{align*}
    For the second term, we have 
    \begin{align*}
        \begin{split}
            \nrm{(\spadesuit_2)}_{C^\vphi}
            & \aleq \nrm{f'_j}_{C^0} \left( \sum_{k=3}^6 \frac{n_k}{n+1} \nrm{\calH[\calT_k^{1,0,\cdots,n_k-1,\cdots}[f]]}_{C^\vphi} \right) \\
            & \quad + \nrm{f'_j}_{C^\vphi} \left( \sum_{k=3}^6 \frac{n_k}{n+1} \nrm{\calH[\calT_k^{1,0,\cdots,n_k-1,\cdots}[f]]}_{C^0} \right) \\
            & \aleq (M\alp_0) a_{n+1} + M \cdot \left( \sum_{k=3}^6 \frac{n_k}{n+1} \nrm{\calH[\calT_k^{1,0,\cdots,n_k-1,\cdots}[f]]}_{C^0} \right)
        \end{split}
    \end{align*}
    For the third term, we have that
    \begin{align*}
        \nrm{(\spadesuit_3)}_{C^\vphi}
        & \aleq \nrm{f'_j}_{C^\vphi} \nrm{\calH[\calT_1^{0,0,n_3,\cdots,n_6}]}_{C^\vphi} \\
        & \aleq M \cdot \rho \alp_0^{C(\tht)} (M\alp_1)^{n} \\
        & \aleq \rho \alp_0^{C(\tht)} (M\alp_1)^{n+1}
    \end{align*}
    For the fourth term, 
    \begin{align*}
        \begin{split}
            \nrm{(\spadesuit_4)}_{C^\vphi} 
            & \aleq \nrm{f'_j}_{C^0} \left( \sum_{k=3}^6 \frac{n_k}{n+1} \nrm{\calH[\calT_k^{1,0,\cdots,n_k-1,\cdots}[f]]}_{C^\vphi} \right) \\
            & \quad + \nrm{f'_j}_{C^\vphi} \left( \sum_{k=3}^6 \frac{n_k}{n+1} \nrm{\calH[\calT_k^{1,0,\cdots,n_k-1,\cdots}[f]]}_{C^0} \right) \\
            & \aleq (M\alp_0) a_{n+1} + M \cdot \left( \sum_{k=3}^6 \frac{n_k}{n+1} \nrm{\calH[\calT_k^{1,0,\cdots,n_k-1,\cdots}[f]]}_{C^0} \right)
        \end{split}
    \end{align*} 
    For the fifth term, 
    \begin{align*}
        \begin{split}
            \nrm{(\spadesuit_5)}_{C^\vphi} 
            & \aleq \nrm{\calH[f'_j]}_{C^\vphi} \nrm{\calT_0^{1,0,n_3,\cdots,n_6}}_{C^\vphi} \\
            & \aleq M \cdot \rho\alp_0^{C(\tht)} (M\alp_1)^{n} \\
            & \aleq \rho \alp_0^{C(\tht)} (M\alp_1)^{n+1}
        \end{split}
    \end{align*}
    Lastly, we have the estimates on the remainder term.
    \begin{align*}
        \begin{split}
            \nrm{J_{\mathrm{rem}}}_{C^\vphi} 
            & \aleq \nrm{f'_j}_{C^\vphi} \nrm{\calH_1[\Pi\calQ[f]^{1,0,n_3,\cdots,n_6}]}_{C^\vphi} \\
            & \aleq M \cdot \nrm{\calH_1[\calQ[f_1]]}_{C^\vphi} \cdot \nrm{\calH_1[\Pi\calQ[f]]^{0,0,n_3,\cdots,n_6}}_{C^\vphi} \\
            & \aleq M \cdot \rho \alp_0^{-\tht} \cdot  \alp_0^{C(\tht)} (M\alp_0)^n \\
            & \aleq \rho \alp_0^{C(\tht)} (M\alp_1)^{n+1}
        \end{split}
    \end{align*}
    In summary, we obtain the following inequality for $a_n$:
    \begin{align*}
        a_{n+2} \aleq (M\alp_0) a_{n+1} + \rho \alp_0^{C(\tht)} (M\alp_0)^{n+1}
    \end{align*}
    We also have the initial condition $a_2$ by using that $\calT_j^{1,0,0\cdots,0}[f] = \calA[f'_j,\calQ[f_1]]$ : 
    \begin{align*}
        \begin{split}
            a_1 
            & = \calA[f'_j,\calQ[f_1]] \\
            & \aleq \nrm{\calH[f'_j]}_{C^\vphi} \nrm{\calQ[f_1]}_{C^\vphi} + \nrm{f'_j}_{C^\vphi} \nrm{\calH_2[\calQ[f_1]]}_{C^\vphi} + \nrm{f'_j}_{C^\vphi} \nrm{\calQ[f_1]}_{C^\vphi} \\
            & \aleq \rho M^1 \alp_0^{C(\tht)}
        \end{split}
    \end{align*}
    Finally we obtain the desired bounds on $a_n$: 
    \begin{align*}
        a_n \aleq \rho \alp_0^{C(\tht)} (M\alp_1)^{n+1}
    \end{align*}
\end{proof}
\addtocontents{toc}{\protect\setcounter{tocdepth}{2}}

\section{On the Global Well-Posedness} \label{Sec:Discussions} 


    The problem of global well-posedness for vortex patches in $C^{1,\vphi}$ currently seems to be out of reach. In this section, we discuss this topic, starting with the two seminal proofs of global well-posedness for classical vortex patches: Bertozzi-Constantin \cite{BC93} and Chemin \cite{Che93}.
    
    We first outline the approach from \cite{BC93} and highlight the difficulties in applying it to the $C^{1,\vphi}$ setting. Their proof for $C^{1,\alp}$ vortex patches, following the notation in \cite{MB02}, relies on four key ingredients:
    \begin{itemize}
        \item Local well-posedness of $C^{1,\alp}$.
        \item A characterization of the patch solution $\omg(t) = \mathbbm{1}_{\Omg(t)}$ via a level set function $\calL(t) \in C^{1,\alp}(\bbR^2)$ such that 
        \begin{align*}
            \Omg(t) = \{ x \in \bbR^2 \mid \calL(t,x) > 0 \} \quad \forall t \geq 0,
        \end{align*} 
        which satisfies the transport equation
        \begin{align} \label{eq:evol of L}
            \rd_t\calL + u \cdot \nabla \calL = 0.
        \end{align} 
        \item A time-frozen, a priori Lipschitz estimate for the velocity $u$
        \begin{align*}
            \nrm{\nabla u}_{L^\infty} \aleq C(1+\log(1+\dlt)),
        \end{align*}
        derived from the \textit{Geometric Lemma}
        \begin{align*}
            |R_\rho(x)| \aleq \left( \frac{d(x,\rd\Omg)}{\rho} + \left( \frac{\rho}{\dlt|x|} \right)^\alp \right).
        \end{align*}
        \item A corresponding time-frozen $C^\alp$ bound on the directional derivative $\nabla u \cdot \nabla^\perp \calL$
        \begin{align*}
            \nrm{\nabla u \cdot \nabla^\perp \calL}_{C^\alp} \aleq \nrm{\nabla u}_{L^\infty} \nrm{\nabla^\perp \calL}_{C^\alp}.
        \end{align*}
    \end{itemize}
    For a detailed explanation, we refer the reader to Chapter 8.3.3 of \cite{MB02} and Chapter 2.3 of \cite{ElgindiJeongMAMS}. A direct application of this strategy to the $C^{1,\vphi}$ space faces two fundamental difficulties. The first, which we will now discuss, is the instantaneous loss of regularity.

    This loss of regularity for the boundary complicates the local well-posedness argument for the level set function $\calL$. To illustrate, we begin by introducing the $C^{1,\alp}$ case. In this setting, a local solution $\Omg(t)$ remains a $C^{1,\alp}$ domain, ensuring the corresponding velocity $u$ remains Lipschitz. Since $\nabla^\perp \calL$ satisfies  
    \begin{align*}
        \rd_t \nabla^\perp \calL + u\cdot\nabla(\nabla^\perp \calL) = \nabla u \cdot \nabla^\perp \calL,    
    \end{align*}
    applying the closed directional $C^{\alp}$ bound on the right-hand side allows one to show that the solution $\calL$ to \eqref{eq:evol of L} remains in $C^\alp$ locally in time.

    In $C^{1,\vphi}$ cases, however, our local existence result (Theorem \ref{Thm2}) establishes that an initial $C^{1,\vphi}$ domain $\Omg_0  $ evolves into a $C^{1,\tld{\vphi}}$ domain for $t>0$ where $\tld{\vphi}$ is the induced modulus from \eqref{eq:def_induced modulus}. Due to this instantaneous loss of regularity, we cannot guarantee that the solution $\calL$ of \eqref{eq:evol of L} remains $C^{1,\vphi}$ even if the two time-frozen estimates are available for the $C^{1,\vphi}$ domain. This precludes a direct extension of the level set function approach. 
    
    In the spirit of the \eqref{eq:HCDE}, one might try to overcome this difficulties by considering a system with an additional function $\calK[\calL]$, where $\calK$ is a singular integral operator such as $\calK = \calH_1\calH_2$:
    \begin{align*}
        \begin{cases}
            & \rd_t \calL + u \cdot \nabla \calL = 0 \\
            & \rd_t \calK[\calL] + \calK[u \cdot \nabla \calL] = 0.
        \end{cases}
    \end{align*}
    However, the two-dimensional geometric interpretation of $\calH[\gma]$ and the relationship between $\calH[\gma]$ and an auxiliary function $\calK[\calL]$ remain unclear. Therefore, this approach also faces significant obstacles.

    

    The second fundamental difficulty, which is linked to the unboundedness of the singular integral operators, arises when examining the time-frozen estimates. Let the vorticity be $\omg = \mathbbm{1}_\Omg$, with a corresponding level set function $\calL$ and velocity $u = -\nabla^\perp (-\Dlt)^{-1} \omg$. A Lipschitz estimate for the velocity can indeed be obtained. Specifically, one can establish a geometric lemma:
    \begin{align*}
        H^1[R_\rho (x_0)] \aleq_\vphi \left( \frac{d(x_0)}{\rho} + \vphi\left( \frac{\rho}{\dlt} \right)\right),
    \end{align*}
    where $\dlt$ is a length scale defined by $\vphi(\dlt) = \frac{|\nabla\calL|_{L^\infty}}{|\nabla\calL|_{C^\vphi}}$. Using this, we obtain the following velocity estimate:
    \begin{align*}
        \nrm{\nabla u}_{L^\infty} \aleq C(1+ \log (1 + \dlt)). 
    \end{align*}
    However, the difficulty arises when attempting to obtain the directional $C^\vphi$ bound. What we can ultimately obtain is the following:
    \begin{align*}
        \nrm{\nabla u \cdot \nabla^\perp \calL}_{C^{\tld{\vphi}}} \aleq \nrm{\nabla u}_{L^\infty} \nrm{\nabla^\perp \calL}_{C^{\vphi}}.
    \end{align*}
    That is, due to the unboundedness of singular integral, we cannot obtain the bounds on $\nrm{\nabla u \cdot \nabla^\perp \calL}_{C^{\vphi}}$.Therefore, the estimate cannot be closed.

    On the other hand, Chemin proved global well-posedness based on paradifferential calculus. Nevertheless, our function space is not easily characterized using Littlewood-Paley decomposition, which poses an additional obstacle.

\section*{Acknowledgements}
The author would like to express sincere gratitude to Professor In-Jee Jeong for his invaluable guidance and support throughout this work. The author was NRF grant from the Korea government (MSIT), No. 2022R1C1C1011051, RS-2024-00406821. 

\bibliographystyle{alpha}
\bibliography{ref.bib}

\begin{thebibliography}{HMVSZ22}

\bibitem[BC93]{BC93}
Andrea~L Bertozzi and Peter Constantin.
\newblock Global regularity for vortex patches.
\newblock 1993.

\bibitem[Ber91]{Ber91}
Andrea~Louise Bertozzi.
\newblock {\em Existence, uniqueness, and a characterization of solutions to
  the contour dynamics equation}.
\newblock Princeton University, 1991.

\bibitem[BH10]{BH10}
Joseph Biello and John~K Hunter.
\newblock Nonlinear hamiltonian waves with constant frequency and surface waves
  on vorticity discontinuities.
\newblock {\em Communications on pure and applied mathematics}, 63(3):303--336,
  2010.

\bibitem[Bur82]{Bur82}
Jacob Burbea.
\newblock Motions of vortex patches.
\newblock {\em Letters in Mathematical Physics}, 6:1--16, 1982.

\bibitem[CD00]{CD00}
Albert Cohen and Raphael Danchin.
\newblock Multiscale approximation of vortex patches.
\newblock {\em SIAM Journal on Applied Mathematics}, 60(2):477--502, 2000.

\bibitem[Che93]{Che93}
Jean-Yves Chemin.
\newblock Persistance de structures g{\'e}om{\'e}triques dans les fluides
  incompressibles bidimensionnels.
\newblock In {\em Annales scientifiques de l'Ecole normale sup{\'e}rieure},
  volume~26, pages 517--542, 1993.

\bibitem[CJ23]{CJ23}
Dongho Chae and In-Jee Jeong.
\newblock Preservation of log-h{\"o}lder coefficients of the vorticity in the
  transport equation.
\newblock {\em Journal of Differential Equations}, 343:910--918, 2023.

\bibitem[CJS24]{CJS24}
Kyudong Choi, In-Jee Jeong, and Young-Jin Sim.
\newblock On existence of sadovskii vortex patch: A touching pair of symmetric
  counter-rotating uniform vortex.
\newblock {\em arXiv preprint arXiv:2406.11379}, 2024.

\bibitem[CL18]{CL18}
Diego C{\'o}rdoba and Omar Lazar.
\newblock Global well-posedness for the 2d stable muskat problem in h3/2.
\newblock {\em arXiv preprint arXiv:1803.07528}, 2018.

\bibitem[CS99]{CS99}
Jos{\'e}~A Carrillo and Juan Soler.
\newblock On the evolution of a singular vortex patch in a two-dimensional
  incompressible fluid flow.
\newblock {\em Computer physics communications}, 121:244--250, 1999.

\bibitem[CS00]{CS00}
Jos{\'e}~A Carrillo and S~Soler.
\newblock On the evolution of an angle in a vortex patch.
\newblock {\em Journal of Nonlinear Science}, 10:23--47, 2000.

\bibitem[DR88]{DuchonRobertVortexSheet}
Jean Duchon and Raoul Robert.
\newblock Global vortex sheet solutions of euler equations in the plane.
\newblock {\em Journal of Differential Equations}, 73(2):215--224, 1988.

\bibitem[DZ78]{DZ78}
Gary~S Deem and Norman~J Zabusky.
\newblock Vortex waves: Stationary" v states," interactions, recurrence, and
  breaking.
\newblock {\em Physical Review Letters}, 40(13):859, 1978.

\bibitem[EJ23a]{EJ23}
Tarek Elgindi and In-Jee Jeong.
\newblock {\em On singular vortex patches, I: Well-posedness issues}, volume
  283.
\newblock American Mathematical Society, 2023.

\bibitem[EJ23b]{ElgindiJeongMAMS}
Tarek Elgindi and In-Jee Jeong.
\newblock {\em On singular vortex patches, I: Well-posedness issues}, volume
  283.
\newblock American Mathematical Society, 2023.

\bibitem[EJ25]{EJ25}
Tarek~M. Elgindi and Min~Jun Jo.
\newblock Cusp formation in vortex patches, 2025.

\bibitem[EM20]{EM20}
Tarek~M Elgindi and Nader Masmoudi.
\newblock L ill-posedness for a class of equations arising in hydrodynamics.
\newblock {\em Archive for Rational Mechanics and Analysis}, 235(3):1979--2025,
  2020.

\bibitem[GL20]{GL20}
Francisco Gancedo and Omar Lazar.
\newblock Global well-posedness for the three dimensional muskat problem in the
  critical sobolev space.
\newblock {\em arXiv preprint arXiv:2006.01787}, 2020.

\bibitem[GP21]{GP21}
Francisco Gancedo and Neel Patel.
\newblock On the local existence and blow-up for generalized sqg patches.
\newblock {\em Annals of PDE}, 7(1):4, 2021.

\bibitem[HMV13]{HMV13}
Taoufik Hmidi, Joan Mateu, and Joan Verdera.
\newblock Boundary regularity of rotating vortex patches.
\newblock {\em Archive for Rational Mechanics and Analysis}, 209:171--208,
  2013.

\bibitem[HMVSZ22]{HM22}
John~K Hunter, Ryan~C Moreno-Vasquez, Jingyang Shu, and Qingtian Zhang.
\newblock On the approximation of vorticity fronts by the burgers--hilbert
  equation.
\newblock {\em Asymptotic Analysis}, 129(2):141--177, 2022.

\bibitem[HT24]{HT24}
De~Huang and Jiajun Tong.
\newblock Steady contiguous vortex-patch dipole solutions of the 2d
  incompressible euler equation.
\newblock {\em arXiv preprint arXiv:2406.09849}, 2024.

\bibitem[Kha24]{Kha24}
Karim R.~Shikh Khalil.
\newblock On the loss and propagation of modulus of continuity for the
  two-dimensional incompressible euler equations, 2024.

\bibitem[KL23]{KL23}
Alexander Kiselev and Xiaoyutao Luo.
\newblock Illposedness of c 2 vortex patches.
\newblock {\em Archive for Rational Mechanics and Analysis}, 247(3):57, 2023.

\bibitem[Koc02]{Koc02}
Herbert Koch.
\newblock Transport and instability for perfect fluids.
\newblock {\em Mathematische Annalen}, 323(3):491--523, 2002.

\bibitem[Maj86]{Maj86}
Andrew Majda.
\newblock Vorticity and the mathematical theory of incompressible fluid flow.
\newblock {\em Communications on Pure and Applied Mathematics},
  39(S1):S187--S220, 1986.

\bibitem[MBO02]{MB02}
Andrew~J Majda, Andrea~L Bertozzi, and A~Ogawa.
\newblock Vorticity and incompressible flow. cambridge texts in applied
  mathematics.
\newblock {\em Appl. Mech. Rev.}, 55(4):B77--B78, 2002.

\bibitem[Mus34]{Mus34}
Morris Muskat.
\newblock Two fluid systems in porous media. the encroachment of water into an
  oil sand.
\newblock {\em Physics}, 5(9):250--264, 1934.

\bibitem[MW83]{MW83}
Jerrold Marsden and Alan Weinstein.
\newblock Coadjoint orbits, vortices, and clebsch variables for incompressible
  fluids.
\newblock {\em Physica D: Nonlinear Phenomena}, 7(1-3):305--323, 1983.

\bibitem[Sad71]{SAD71}
VS~Sadovskii.
\newblock Vortex regions in a potential stream with a jump of bernoulli's
  constant at the boundary: Pmm vol. 35, no. 5, 1971, pp. 773--779.
\newblock {\em Journal of Applied Mathematics and Mechanics}, 35(5):729--735,
  1971.

\bibitem[Tor04]{Torchinsky}
Alberto Torchinsky.
\newblock {\em Real-variable methods in harmonic analysis}.
\newblock Courier Corporation, 2004.

\bibitem[Yud63]{Yud63}
Vladimir~I Yudovich.
\newblock The flow of a perfect, incompressible liquid through a given region.
\newblock In {\em Soviet Physics Doklady}, volume~7, page 789, 1963.

\bibitem[ZHR79]{ZHR79}
Norman~J Zabusky, MH~Hughes, and KV~Roberts.
\newblock Contour dynamics for the euler equations in two dimensions.
\newblock {\em Journal of computational physics}, 30(1):96--106, 1979.

\end{thebibliography}

\end{document}